\documentclass[1 [leqno,11pt]{amsart}
\usepackage{amssymb, amsmath,amsmath,latexsym,amssymb,amsfonts,amsbsy, amsthm,color}

\numberwithin{equation}{section}

\allowdisplaybreaks[1]

\newtheorem{theorem}{Theorem}[section]
\newtheorem{lemma}[theorem]{Lemma}
\newtheorem{remark}[theorem]{Remark}

\newtheorem{proposition}[theorem]{Proposition}

\numberwithin{equation}{section}

\newcommand{\na}{\nabla}
\newcommand{\dd}{\mathrm{d}}

\newcommand{\R}{\mathbb{R}}
\newcommand{\Us}{\mathbf{U}_s}
\newcommand{\U}{\mathbf{U}}
\newcommand{\W}{\mathbf{W}}
\newcommand{\UE}{U^\varepsilon}
\newcommand{\VE}{V^\varepsilon}

\newcommand{\e}{\varepsilon}
\newcommand{\T}{\mathcal{T}}
\newcommand{\X}{\mathbb{X}}
\newcommand{\Y}{\mathbb{Y}}
\newcommand{\Z}{\mathbb{Z}}

\newcommand{\z}{\langle}
\newcommand{\y}{\rangle}
\newcommand{\p}{\partial}
\newcommand{\se}{\sqrt\e}
\newcommand{\RE}{\mathbf{R}}
\newcommand{\F}{\mathbf{F}}
\newcommand{\ub}{\bar{u}}
\newcommand{\vb}{\bar{v}}
\newcommand{\ed}{\eta_\delta}
\newcommand{\ot}{\tilde{\omega}}
\newcommand{\de}{\delta}

\newcommand{\td}{\tilde}
\newcommand{\n}{\ensuremath{\nonumber}}

\newcommand{\yy}{\rangle_{L^2_{x}(y=0)}}

\newcommand{\fsx}{\|_{L^2_y(x=x_0)}}
\newcommand{\fs}{\|_{L^2_{x,y}}}

\newcommand{\jfxy}{\int_0^{x_0}\int_0^\infty}

\newtheorem{Lemma}{Lemma}[section]

\begin{document}

\title[Remarks on the Steady Prandtl Boundary Layer Expansions]
{Remarks on the Steady Prandtl Boundary Layer Expansions}

\author{Chen Gao}
\address{Beijing International Center for Mathematical Research, Peking University, Beijing 100871, China}
\email{gaochen@amss.ac.cn}

\author{Liqun Zhang}
\address{Hua Loo-Keng Key Laboratory of Mathematics, Institute of Mathematics, AMSS, and School of
Mathematical Sciences, UCAS, Beijing 100190, China 
123}
\email{lqzhang@math.ac.cn}

\begin{abstract}
We continue the study of the validity of the Prandtl boundary layer expansions in \cite{GZ}, where by estimating the stream-function of the remainder, we proved if the Euler flow is perturbation of shear flows when the width of domain is small.  In this paper, we obtain a new
derivatives estimate of stream-function away from the boundary layer and then prove the validity
of expansions for any non-shear Euler flow, provided the width of domain is small.
\end{abstract}
\date{}
\maketitle

\section{Introduction}\label{sec:intro}
We consider the stationary incompressible Navier-Stokes equations
\begin{equation}\label{NSE}
\left\{
\begin{aligned}
&U^\e U^\e_X+V^\e U^\e_Y-\e\Delta U^\e+P^\e_X=0,\\
&U^\e V^\e_X+V^\e V^\e_Y-\e\Delta V^\e+P^\e_Y=0,\\
&U^\e_X+V^\e_Y=0,
\end{aligned}
\right.
\end{equation}
posed in a two dimensional domain $\Omega=\{(X,Y):0< X< L, Y>0\}.$ The no-slip boundary conditions are set on the boundary $Y=0$:
$$U^\e(X,0)=0, \hspace{3mm}V^\e(X,0)=0.$$
We are concerned with the asymptotic behavior of solution $[U^\e,V^\e]$ when $\e$ is small. A formal limit $\e\rightarrow 0^+$ should lead to the Euler flow $[U^0,V^0]$ inside $\Omega$:
\begin{equation}\label{Euler}
\left\{
\begin{aligned}
&U^0 U^0_X+V^0 U^0_Y+P^0_X=0,\\
&U^0 V^0_X+V^0 V^0_Y+P^0_Y=0,\\
&U^0_X+V^0_Y=0.
\end{aligned}
\right.
\end{equation}
Naturally, we pose the system (\ref{Euler}) with no penetration boundary condition on $Y=0$:
$$V^0(X,0)=0.$$
Generically, there is a mismatch between the tangential velocities of the Euler flow $U^{0}(X,0)\neq0$ and the prescribed Navier-Stokes flow $\UE(X,0) = 0$ on the boundary, because of the difference of boundary conditions imposed on the two systems.

According to the classical Prandtl boundary layer theory, there exists a thin layer which connect with $\UE(X,0)$ and $U^\e(X,0)$. Precisely, we take the Prandtl's variables: 
\begin{align}\label{var change}
x = X, \qquad y=\frac{Y}{\se}.
\end{align}
In these variables, we express the solution of the NS equations $[\UE,\VE]$ via $[u^\e,v^\e]$ as
\[
\begin{split}
[\UE(X,Y),\VE(X,Y)]=[u^\e(x,y),\se v^\e(x,y)]
\end{split}
\]
in which we note that the scaled normal velocity $v^\e$ is $\frac{1}{\se}$ of the original velocity $\VE$ to satisfy the divergence-free condition. Similarly, $P^\e(X,Y) = p^\e(x,y)$. In these new variables, the Navier-Stokes equations (\ref{NSE}) now read
\begin{equation}\label{P-NSE}
\left\{
\begin{aligned}
u^\e u^\e_x+v^\e u^\e_y+p^\e_x&=u^\e_{yy}+\e u^\e_{xx},\\
\e[u^\e v^\e_x+v^\e v^\e_y]+p^\e_y&=\e[v^\e_{yy}+\e v^\e_{xx}],\\
u^\e_x+v^\e_y&=0.
\end{aligned}
\right.
\end{equation}
Let $\e\rightarrow0$, it leads to the Prandtl equations:
\begin{equation}\label{Prandtl}
\left\{
\begin{aligned}
&u^0_pu^0_{px}+v^0_pu^0_{py}-u^0_{pyy}+p^0_{px}=0,\\
&p^0_{py}=0,\\
&u^0_{px}+v^0_{py}=0,\\
\end{aligned}
\right.
\end{equation}
with $u^0_p|_{y=0}=0$. Prandtl hypothesized that when viscosity $\e$ is small the Navier-Stokes flow can be approximately decomposed into two parts:
\begin{align}\label{p-exp}
\begin{aligned}
&\UE(X,Y)\approx u^0_e(X,Y)-u^0_e(X,0)+u^0_p(X,\frac{Y}{\se}),\\
&\VE(X,Y)\approx v^0_e(X,Y)+\se v^0_p(X,\frac{Y}{\se}),
\end{aligned}
\end{align}
in which $[u^0_e,v^0_e]=[U^0,V^0]$ denotes the Euler flow. 

The verification of the viscosity vanishing limits is a challenging problem in general. There are lots of studies in recent year. For non-stationary case, the problem in the analytic case was proved in \cite{Sam1}, \cite{Sam2} and \cite{WWZ17}. In 2014, Maekawa \cite{Mae} proved the convergence under the assumption on the initial vorticity vanishing in the neighbourhood of boundary. Fei, Tao and Zhang \cite{TaoTao} generalized this result to 3D case by energy methods. In \cite{GMM18}, Gerard-Varet, Maekawa and Masmoudi established the Gevrey stability for Prandtl type shear flows. Chen, Wu and Zhang gave a new proof of Gevrey stability for steady profile by resolvent estimate method in \cite{CWZ}. Later, Gerard-Varet, Maekawa and Masmoudi \cite{GMM2020} showed the Prandtl expansion around concave boundary layer in Gevrey space. There are some results of instablity in Sobolev space, cf. \cite{GGN16}-\cite{GN18}.
 
For the steady case, the problem is considered in Sobolev space. In this situation, the Euler flow $[u^0_e,v^0_e]$ in expansion (\ref{p-exp}) is always shear flow $[u^0_e(Y),0]$.  An important progress was made by Guo, Nguyen \cite{GN} for Prandtl boundary layer expansions for the steady Navier-Stokes flows over a moving plate. Then Iyer \cite{I17} extended this result into a rotating disk. They considered the Euler flow is shear or rotating shear, Prandtl profile is strictly positive and the width of region or the angle of sector is small. Later, Iyer \cite{I} generalized the result in \cite{GN} for the perturbation of shear flow. After that, Iyer in \cite{IPU1}, \cite{IPU2} and \cite{IPU3} justified the steady Prandtl expansion over a moving plane in $(0,\infty)\times(0,\infty)$ under the assumption of constant shear Euler flow and smallness of the boundary layer profile. In 2018, a significant work by Guo and Iyer \cite{GI} showed the convergence result for no-slip boundary conditions in shear Euler flows in the case the width of the region $L$ is small. Inspired by the methods in this work, we introduced a new quantity, the quotient of steam-function and the approximate solution, in \cite{GZ}. By estimating that quantity, we justified the validity of expansion (\ref{p-exp}) for the perturbation of shear Euler flow when $L$ is small. Moreover, we showed when the Euler flow is shear and Prandtl profile is in monotonic class, (\ref{p-exp}) is right even $L$ is large. Recently, Iyer and Masmoudi in \cite{IM} also estimated the quantity introduced in \cite{GZ} to show the stability of the Prandtl expansion in domain $(0,\infty)\times(0,\infty)$. In that work, the Euler flow is considered as special shear flow $[1,0]$ and the solution of Prandtl equation is famous Blasius flow in monotonic class. There are also the stability results for Prandtl type shear flow of Navier-Stokes equations with force term in $X$-periodic domain cf. \cite{Gerard pra} and \cite{CWZ2021}.

In this paper, we continue our earlier work in \cite{GZ}. We assume that the outside Euler flow $[u^0_e(X,Y),v^0_e(X,Y)]$ satisfying the following hypothesis:
\begin{align}\label{Eul profile}
\begin{aligned}
&0<c_0\leqslant u^0_e\leqslant C_0<\infty,\\
&\|\z Y\y^k\nabla^m [u^0_e,v^0_e]\|_{L^\infty}<\infty \text{ for }m\geqslant1.
\end{aligned}
\end{align}
Here $\z Y\y=Y+1$ and $k$ is a large constant. The special case for Euler flow is shear flow $[u^0_e(Y),0]$ which discussed in \cite{GI} and \cite{GZ}.

We consider the Prandtl equations with the positive data.
\begin{equation}\label{pra them}
\left\{
\begin{aligned}
&u^0_pu^0_{px}+v_p^0u^0_{py}-u^0_{pyy}+p^0_{px}=0,\hspace{3 mm}p^0_{py}=0,\hspace{3 mm}u_{px}^0+v_{py}^0=0,\hspace{3 mm}(x,y)\in(0,L)\times\mathbb{R}_+,\\
&u^0_p|_{x = 0} = U^0_P(y), \hspace{5 mm} u^0_p|_{y = 0} = v^0_p|_{y = 0} = 0, \hspace{5 mm} u^0_p|_{y \uparrow \infty} = u^0_e|_{Y = 0}.
\end{aligned}
\right.
\end{equation}
$U^0_P$ is a prescribing smooth function such that
\begin{align}
\begin{aligned} \label{pos}
& U^0_P > 0 \text{ for } y > 0, \hspace{3 mm} \partial_y U^0_P(0) > 0, \hspace{3 mm} \partial_y^2 U^0_P-u^0_e(x,0)u^0_{ex}(x,0) \sim y^2 \text{ near } y = 0,\\
&\partial_y^m \{U^0_P - u^0_e(x,0)\} \text { decay fast for any }m\geqslant0.
\end{aligned}
\end{align}
By the classical result in \cite{Oleinik}, under above conditions on $U^0_P$, if $L$ is small enough, equations (\ref{pra them}) admit a classical solution $[u^0_p,v^0_p]$ satisfying:
\begin{equation}\label{Pra profile}
\begin{aligned}
&u^0_p > 0 \text{ for } y > 0, \hspace{3mm} u^0_{py}|_{y=0}>0,\\
&\nabla^m \{u^0_p - u^0_e(0)\} \text { decay fast as $y\rightarrow\infty$ for any }m\geqslant0.
\end{aligned}
\end{equation}

Now we state our main result.
\begin{theorem}\label{main}  Assume the Euler flow $[u^0_e,v^0_e]$ satisfies (\ref{Eul profile}), the Prandtl profile satisfies (\ref{Pra profile}), $L$ is a constant small enough,
\noindent then there exist $\e_0(L)>0$ depending on $L$, such that for $0<\e\leqslant\e_0$, equations (\ref{NSE}) admits a solution $[U^\e,V^\e]\in W^{2,2}(\Omega)$, satisfying:
\begin{align}
\begin{aligned}
&\|U^\e-u^0_e+u^0_e|_{Y=0}-u^0_p\|_{L^\infty}\leqslant C\se,\\
&\hspace{1cm}\|V^\e-v^0_e\|_{L^\infty}\leqslant C\se,
\end{aligned}
\end{align}
with the following boundary conditions:
\begin{equation}\label{Boundary C}
\begin{aligned}
&[U^\e, V^\e]|_{Y=0}=0,\\
&[U^\e, V^\e]|_{X=0}=[u^0_e(0,Y)-u^0_e(0,0)+u^0_p(0,\frac{Y}{\se})+\se a_0, v^0_e(0,Y)+\se b_0],\\
&[U^\e, V^\e]_{X=L}=[u^0_e(L,Y)-u^0_e(L,0)+u^0_p(L,\frac{Y}{\se})+\se a_L, v^0_e(L,Y)+\se b_L].
\end{aligned}
\end{equation}
Here $C$ is a constant independent of $L$ and $\e$,
\begin{equation*}
\begin{aligned}
&a_0(Y)=u^1_e(0,Y)+u^1_b(0,\frac{Y}{\se})+\se u^2_e(0,Y)+\se \hat{u}^2_b(0,\frac{Y}{\se}),\\
&a_L(Y)=u^1_e(L,Y)+u^1_b(L,\frac{Y}{\se})+\se u^2_e(L,Y)+\se \hat{u}^2_b(L,\frac{Y}{\se}),\\
&b_0(Y)=v^0_b(0,\frac{Y}{\se})+v^1_e(0,Y)+\se v^1_b(0,\frac{Y}{\se})+\se v^2_e(0,Y)+\e \hat{v}^2_b(0,\frac{Y}{\se}),\\
&b_L(Y)= v^0_b(L,\frac{Y}{\se})+v^1_e(L,Y)+\se v^1_b(L,\frac{Y}{\se})+\se v^2_e(L,Y)+\e \hat{v}^2_b(L,\frac{Y}{\se}),\\
\end{aligned}
\end{equation*}
are smooth functions constructed in Proposition \ref{construct}.
\end{theorem}

Unlike the assumptions of shear flows or their perturbation in previous works, this theorem shows the expansions (\ref{p-exp}) for any non-shear Euler flow which $u^0_e$ is strictly positive. This result and the second result in \cite{GZ} coincide with the classical results of Oleinik and Samokhin \cite{Oleinik} for solutions of Prandtl's equation. They show the local well-posedness ($L$ is small) of Prandtl equation for any $u^0_e(X,0)$ is strictly positive and the global well-posedness ($L$ is any constant) for $u^0_{eX}(X,0)\geqslant 0$. And we show the Prandtl expansions for non-shear Euler flow when $L\ll1$ in this paper and shear Euler flow when $L$ is any given constant in \cite{GZ}. 

To prove the theorem, we first construct the approximate solutions $\U_s=[U_s,V_s]$ of Navier-Stokes equations, which is similar to \cite{GZ}. The main difficulty is estimating the remainders $\U:=\U^\e-\U_s$, where $\U$ satisfies the following linearized Navier-Stokes equations:
\begin{align}
-\e\Delta \U+\Us\cdot\nabla\U+\U\cdot\nabla\Us+\nabla P=\F.
\end{align}
But when Euler flows is non-shear, $V_s\approx v^0_e$ is not small. We notice that $v^0_e(X,0)=0$, so $v^0_e$ is small near the boundary $\{(X,Y)|Y=0\}$. It leads us to estimating the stream-function away from the boundary layers. We estimate the derivatives of stream-function in  outer area, i.e. $\{Y\geqslant\delta>0\}$, and combine this estimate with some estimates we obtained in \cite{GZ}, to show the stream-function can be dominating by $\F$, which essentially leads to the proof the theorem. 

This paper is organized as follows: In Section 2, we show the main profile of the approximation solution. In Section 3, we estimate the stream-function of remainder. In Section 4, we prove the main theorem. The construction of the high-order approximation solutions is in Appendix.

\section{Construction of the approximate solution}
The construct of the approximate solutions is similar to \cite{GZ}. We will need higher order expansions, as compared to (\ref{p-exp}), in order to control the remainder. Actually, the approximate solutions of the Navier-Stokes equations are as the following form:
\begin{equation}\label{proflie}
\begin{aligned}
U^\e(X,Y)\approx&u^0_e(X,Y)+u^0_b(X,\frac{Y}{\se})+\se[u^1_e(X,Y)+u^1_b(X,\frac{Y}{\se})]\\
         &+\e[u^2_e(X,Y)+u^2_b(X,\frac{Y}{\se})],\\
V^\e(X,Y)\approx&v^0_e(X,Y)+\se[v^0_b(X,\frac{Y}{\se})+v^1_e(X,Y)]+\e[v^1_b(X,\frac{Y}{\se})+v^2_e(X,Y)]\\
         &+\e^\frac{3}{2}v^2_b(X,\frac{Y}{\se}),\\
P^\e(X,Y)\approx&p^0_e(X,Y)+p^0_b(X,\frac{Y}{\se})+\se[p^1_e(X,Y)+p^1_b(X,\frac{Y}{\se})]\\
         &+\e[p^2_e(X,Y)+p^2_b(X,\frac{Y}{\se})]+\e^\frac{3}{2}p^3_b(X,\frac{Y}{\se}),
\end{aligned}
\end{equation}
in which $[u^j_e,v^j_e,p^j_e]$ and $[u^j_b,v^j_b,p^j_b]$, with $j=0,1,2,$ denoting the Euler profiles and boundary layer profiles, respectively. Here, we note that these profile solutions also depend on $\e$. And the Euler flows are always evaluated at $(X,Y)$, whereas the boundary layer profiles are at $(X,\frac{Y}{\se})$.

{\bf Notation.}
For convenience, we will introduce some notation here. we write $$\z\cdot,\cdot\y=\z\cdot,\cdot\y_{L^2_{X,Y}}, $$ $$\z\cdot,\cdot\y_{Y=0}=\z\cdot,\cdot\y_{L^2_X(Y=0)},$$ $$\|\cdot\|=\|\cdot\|_{L^2_{X,Y}}$$ and $$\|\cdot\|_{\infty}=\|\cdot\|_{L^\infty_{X,Y}}=\|\cdot\|_{L^\infty_{x,y}}.$$ We denote $a\lesssim b$ which means there exist a positive constant $C_0$, s.t. $a\leqslant C_0b$, here $C_0$ is independent on $\se$ and $L$. And we write $a=O\big(b\big)$ as $|a|\lesssim b$.

\subsection{The leading order of approximate solution}
Recall the Euler flow $[u^0_e,v^0_e]$. Let $\psi$ be the stream-function of $[u^0_e,v^0_e]$
\begin{align}
\psi(X,Y):=\int_0^Yu^0_e(X,Y')\dd Y', \hspace{3 mm} \psi_Y=u^0_e, \hspace{3 mm} \psi_X=-v^0_e,
\end{align}
then Euler equations (\ref{Euler}) are equivalent to:
\begin{align}\label{F-Euler}
\Delta\psi=F_e(\psi).
\end{align}
From the assumptions in (\ref{Eul profile}), we can know that $F_e$ together with sufficiently many derivatives are bounded and decaying in its argument.

For Prandtl equations, there is a famous result due to Oleinik \cite{Oleinik}:
\begin{proposition}[Oleinik]   \label{Oleinik} Assume boundary data  $U^0_P \in C^\infty$ satisfies (\ref{pos}), then for some $L > 0$, equations (\ref{pra them}) exists a solution $[u^0_p, v^0_p]$, satisfying, for some $y_0, m_0 > 0$,
\begin{equation}\label{OL}
\begin{aligned}
&\sup_{(0,L)\times(0,\infty)} |u^0_p, \p_y u^0_p, \p_{yy}u^0_p, \p_x u^0_p| \lesssim 1, \\
&\inf_{(0,L)\times (0, y_0)} \p_y u^0_p > m_0 > 0,\\
&\hspace{3mm} u^0_p > 0,\hspace{1mm}\text{ for } y > 0.
\end{aligned}
\end{equation}
\end{proposition}
In fact, if $U^0_P$ satisfies high order parabolic compatibility conditions at the corner $(0,0)$, then $[u^0_p,v^0_p]$ are smooth enough. The compatibility conditions are discussed in our previous works \cite{GZ}. Following the proof of Oleinik in \cite{Oleinik}, we have:
\begin{lemma}\label{pra0}
If $U^0_P \in C^\infty$ satisfies (\ref{pos}) and high order parabolic compatibility conditions, then
\begin{align}
\|\z y\y^M\nabla^k(u^0_p(x,y)-u^0_e(x,0))\|_{\infty}\lesssim1, \hspace{2mm}\text{  for  }\hspace{2mm} 0\leqslant k \leqslant K,
\end{align}
here $K$ and $M$ are constants.
\end{lemma}

After we solved Prandtl's equation (\ref{pra them}), we set $$u^0_b(x,y)=u^0_p(x,y)-u^0_e(x,0),$$ $$v^0_b(x,y)=\int_y^\infty u^0_{bx}(x,y')\dd y'$$ and $p^0_b=0.$ The construction of high-order approximate solutions $[u^1_e,v^1_e]$, $[u^2_e,v^2_e]$, $[u^1_b,v^1_b]$ and $[\hat{u}^2_b,\hat{v}^2_b]$ is in Appendix. We denote the $[U_s,V_s]$ as the following:
\begin{equation}
\begin{aligned}
U_s(X,Y)=&u^0_e(X,Y)+u^0_b(X,\frac{Y}{\se})+\se[u^1_e(X,Y)+u^1_b(X,\frac{Y}{\se})]\\
          &+\e[u^2_e(X,Y)+\hat{u}^2_b(X,\frac{Y}{\se})],\\
V_s(X,Y)=&v^0_e(X,Y)+\se[v^0_b(X,\frac{Y}{\se})+v^1_e(X,Y)]+\e[v^1_b(X,\frac{Y}{\se})+v^2_e(X,Y)]\\
          &+\e^\frac{3}{2}\hat{v}^2_b(X,\frac{Y}{\se}),\\
P_s(X,Y)=&p^0_e(X,Y)+p^0_b(X,\frac{Y}{\se})+\se[p^1_e(X,Y)+p^1_b(X,\frac{Y}{\se})]\\
          &+\e[p^2_e(X,Y)+p^2_b(X,\frac{Y}{\se})]+\e^\frac{3}{2}p^3_b(X,\frac{Y}{\se}).
\end{aligned}
\end{equation}
Then the errors
\begin{equation}\label{R1,R2}
\begin{aligned}
&R_1:=U_sU_{sX}+V_sU_{sY}-\e\Delta U_s+P_{sX},\\
&R_2:=U_sV_{sX}+V_sV_{sY}-\e\Delta V_s+P_{sY},
\end{aligned}
\end{equation}
satisfy
\begin{equation}
\begin{aligned}\label{R order}
\|R_1\|+\|R_2\|\lesssim\e^{\frac{3}{2}}.
\end{aligned}
\end{equation}

Next we show the main profile of $[U_s,V_s]$.

\subsection{The main profile of approximate solution}
In order to obtain the estimates of the remainder, we need to know more information about $[U_s,V_s]$.
\[
\begin{split}
&U_s(X,Y)= u^0_e(X,Y)+u^0_b(X,\frac{Y}{\se})+O\big(\se\big),\\
&V_s(X,Y)= v^0_e(X,Y)+\se (v^0_b(X,\frac{Y}{\se})+v^1_e(X,Y))+O\big(\e\big).
\end{split}
\]
Since $u^0_e(X,Y)$ is strictly positive, $u^0_p(X,\frac{Y}{\se})>0$ for $Y>0$, $u^0_p(X,0)=0$ and $u^0_{py}(X,0)>0$, when $\frac{Y}{\se}\leqslant1$,
\[
\begin{split}
u^0_e(X,Y)+u^0_b(X,\frac{Y}{\se})=u^0_e(X,Y)-u^0_e(X,0)+u^0_p(X,\frac{Y}{\se})\gtrsim-Y+\frac{Y}{\se}\gtrsim\frac{Y}{\se},
\end{split}
\]
when $1\leqslant\frac{Y}{\se}\leqslant \e^{-\frac{1}{4}}$,
\[
\begin{split}
u^0_e(X,Y)+u^0_b(X,\frac{Y}{\se})=u^0_e(X,Y)-u^0_e(X,0)+u^0_p(X,\frac{Y}{\se})\gtrsim-Y+1\gtrsim1,
\end{split}
\]
when $\frac{Y}{\se}\gtrsim \e^{-\frac{1}{4}}$,
\[
\begin{split}
u^0_e(X,Y)+u^0_b(X,\frac{Y}{\se})\gtrsim1.
\end{split}
\]
So $U_s\sim\frac{Y}{\se},$ when $Y\leqslant\se,$ and $U_s\sim 1,$ when $Y\geqslant\se.$

One can easily see for $i,j\geqslant0$,
\begin{align}\label{main pro}
\begin{aligned}
&\|\p^j_XU_s\|_\infty\lesssim1,\hspace{2mm}\|\p^j_XV_s\|_\infty\lesssim1,\\
&\se^i\|Y^j\p^{i+j}_YU_s\|_\infty\lesssim\se^i\|Y^j\p^{i+j}_Yu^0_e\|_\infty+\se^i\|\frac{y^j}{\se^i}\p^{i+j}_yu^0_b\|_\infty+O\big(\se\big)\lesssim1,\\
&\se^i\|Y^j\p^{i+j+1}_YV_s\|_\infty\lesssim\se^i\|Y^j\p^{i+j+1}_Yu^0_e\|_\infty+\se^i\|\frac{y^j}{\se^i}\p^{i+j+1}_yu^0_b\|_\infty+O\big(\se\big)\lesssim1.
\end{aligned}
\end{align}

Let $\delta$ be a positive constant satisfying $\de\geqslant\e^\frac{1}{4}$, when $Y\leqslant\de$, we have 
\begin{align}
\begin{aligned}
|V_s|\lesssim \de U_s.
\end{aligned}
\end{align}
In fact, when $Y\leqslant\se$,
\[
\begin{split}
|v^0_e(X,Y)|\lesssim Y=\se \frac{Y}{\se}\lesssim \se U_s\lesssim \de U_s,
\end{split}
\]
when $\se\leqslant Y\leqslant\de$,
\[
\begin{split}
|v^0_e(X,Y)|\lesssim Y\lesssim Y U_s\lesssim \de U_s.
\end{split}
\]

When $Y\geqslant\delta$, for any $j\geqslant0$, 
\begin{align}\label{away b}
\begin{aligned}
&\p^j_YU_s(X,Y)= \p^j_Yu^0_e(X,Y)+O\big(\se\big),\\
&\p^j_YV_s(X,Y)= \p^j_Yv^0_e(X,Y)+O\big(\se\big).
\end{aligned}
\end{align}
In fact since $u^j_b(X,y), v^j_b(X,y)$ decay fast when $y\rightarrow\infty$, for some large $M$,
\[
\begin{split}
 |u^j_b(X,\frac{Y}{\se})|\lesssim \big(\frac{Y}{\se}\big)^{-M}\lesssim \se^M\delta^{-M}\lesssim \e^\frac{M}{4},\\
 |v^j_b(X,\frac{Y}{\se})|\lesssim \big(\frac{Y}{\se}\big)^{-M}\lesssim \se^M\delta^{-M}\lesssim \e^\frac{M}{4}.
\end{split}
\]

\section{Estimates of the remainder}
In this section, we show the estimates of the remainder. Let
\begin{align}
\UE=U_s+U,\hspace{3 mm}\VE=V_s+V.
\end{align}
Then
\begin{equation}
\left\{
\begin{aligned}
&U_sU_X+U_{sX}U+V_sU_Y+U_{sY}V-\e\Delta U+P_X=-\{R_1+UU_X+VU_Y\},\\
&U_sV_X+V_{sX}U+V_sV_Y+V_{sY}V-\e\Delta V+P_Y=-\{R_2+UV_X+VV_Y\},\\
&U_X+V_Y=0.
\end{aligned}
\right.
\end{equation}
We consider the linearized equations:
\begin{equation}\label{LNSE}
\left\{
\begin{aligned}
&U_sU_X+U_{sX}U+V_sU_Y+U_{sY}V-\e\Delta U+P_X=F_1,\\
&U_sV_X+V_{sX}U+V_sV_Y+V_{sY}V-\e\Delta V+P_Y=F_2,\\
&U_X+V_Y=0.
\end{aligned}
\right.
\end{equation}
Our critical estimates is the following proposition.
\begin{proposition}\label{Prop}
Under the assumptions in theorem\ref{main}, if $$[U,V]\in W^{2,2}(\Omega)\cap W^{1,2}_0(\Omega)$$ satisfies the equations (\ref{LNSE}), \noindent then
\begin{align}
\|\se U_X,\se U_Y,\se V_X, \se V_Y, U, V\|\leqslant C(\|F_1\|+\|F_2\|).
\end{align}
\end{proposition}
Let $\Phi$ be the stream-function of $U,V$, that is, $\Phi_X=-V, \Phi_Y=U$, We can solve the stream-function by this way 
\begin{align}
\Phi(X,Y)=\int_0^Y U(X,Y')\dd Y'.
\end{align}
According to the boundary conditions $[U,V]_{\Omega}=0$ and $U_X+V_Y=0$, We have 
\begin{align}
\begin{aligned}
&\Phi|_{X=0}=\Phi|_{X=L}=\Phi|_{Y=0}=0,\\
&\Phi_X|_{X=0}=\Phi_X|_{X=L}=\Phi_Y|_{Y=0}=0.
\end{aligned}
\end{align}
If $U,V\in L^2(\Omega)$, then $\Phi_Y,\Phi_X\in L^2(\Omega)$.

Because $[U,V]$ satisfy the equations (\ref{LNSE}), We can deduce the equation of stream function
\begin{equation}\label{Phi-Euq}
\left\{
\begin{aligned}
&U_s\Delta\Phi_X-\Phi_X\Delta U_s-\e\Delta^2\Phi+V_s\Delta\Phi_Y-\Phi_Y\Delta V_s=\partial_Y F_1-\partial_X F_2,\\
&\Phi|_{X=0}=\Phi|_{X=L}=\Phi|_{Y=0}=\Phi_X|_{X=0}=\Phi_X|_{X=L}=\Phi_Y|_{Y=0}=0.
\end{aligned}
\right.
\end{equation}
It is the fourth-order elliptic equation for $\Phi$, the boundary conditions are about $\Phi$ and its derivatives.

Let $G=\frac{\Phi}{U_s}$, $G$ and $\Phi$ satisfy
\begin{equation}\label{G-Euq}
\left\{
\begin{aligned}
&\partial_{XX}[U_s^2G_X]+\partial_{XY}[U_s^2G_Y]-\e\Delta^2\Phi+R[\Phi]=\partial_Y F_1-\partial_X F_2,\\
&G|_{X=0}=G|_{X=L}=G|_{Y=0}=G_X|_{X=0}=G_X|_{X=L}=0.
\end{aligned}
\right.
\end{equation}
where $R[\Phi]=V_s\Delta\Phi_Y-U_{sX}\Delta\Phi-\Phi_Y\Delta V_s+\Phi\Delta U_{sX}$.
We define two norms of $G$:
\begin{equation}\label{space XY}
\begin{aligned}
\|G\|_\X^2&:=\|U_sG_X\|^2+\|U_sG_Y\|^2,\\
\|G\|_\Y^2&:=\e\{\|\sqrt U_sG_{XX}\|^2+2\|\sqrt U_sG_{XY}\|^2+\|\sqrt U_sG_{YY}\|^2\}.
\end{aligned}
\end{equation}

We start with the following Hardy-type's inequality in \cite{GZ}.
\begin{Lemma}\label{Hardy}
If $H\in W^{1,2}(0,\infty)$, $0<\xi\leqslant1$, \noindent then
$$ \|H\|_{L^2_Y}^2\leqslant C\xi\e\|\sqrt {U_s}H_Y\|_{L^2_Y}^2+\frac{C}{\xi^2}\|U_sH\|_{L^2_Y}^2.$$
\end{Lemma}
\begin{proof}
Let $\chi:[0,\infty)\rightarrow [0,1]$ be a smooth cut-off function supported in $[0,2]$, and $\chi|_{[0,1]}=1$,
$$\int_0^\infty H^2\dd Y\lesssim\int_0^\infty H^2\chi^2(\frac{Y}{\xi\se})\dd Y+ \int_0^\infty H^2(1-\chi(\frac{Y}{\xi\se}))^2\dd Y.$$
Recall the leading profile of $U_s$, 
\begin{align*}
U_s\sim\left\{
\begin{aligned}
&\frac{Y}{\se},\hspace{2mm} if \hspace{1mm} Y\leqslant\se, \\ 
&1, \hspace{2mm} if \hspace{1mm} Y\geqslant\se.
\end{aligned}
\right.
\end{align*}
So when $\frac{Y}{\se}\leqslant1$, $1-\chi(\frac{Y}{\xi\se})\lesssim\frac{Y}{\xi\se}\lesssim\frac{U_s}{\xi}$, and when $\frac{Y}{\se}\geqslant1$, $1-\chi(\frac{Y}{\xi\se})\lesssim1\lesssim\frac{U_s}{\xi}$. We have
$$\int_0^\infty H^2(1-\chi(\frac{Y}{\xi\se}))^2\dd Y\lesssim\frac{1}{\xi^2}\int_0^\infty U_s^2H^2\dd Y.$$
And
\[
\begin{aligned}
\int_0^\infty H^2\chi^2(\frac{Y}{\xi\se})\dd Y&=-2\int_0^\infty YHH_Y\chi^2(\frac{Y}{\xi\se})\dd Y-2\int_0^\infty\frac{Y}{\xi\se}H^2\chi'(\frac{Y}{\xi\se})\chi(\frac{Y}{\xi\se})\dd Y\\
                               &\lesssim\int_0^\infty Y^2H_Y^2\chi^2(\frac{Y}{\xi\se})\dd Y+\int_0^\infty({\frac{Y}{\xi\se}})^2H^2|\chi'(\frac{Y}{\xi\se})|\chi(\frac{Y}{\xi\se})\dd Y\\
                               &\lesssim\xi\e\int_0^\infty U_sH_Y^2\dd Y+\frac{1}{\xi^2}\int_0^\infty U_s^2H^2\dd Y.
\end{aligned}
\]
The proof is complete.
\end{proof}

The next lemma is a basic elliptic estimate in \cite{GZ}, it is true even though Euler flows is non-shear.
\begin{Lemma}\label{routine}
Let $G$ be the solution of equation (\ref{G-Euq}), $L>0$, then
\begin{align}\label{rou}
\|G\|_\Y^2\lesssim\|G\|_\X^2+|\z\partial_YF_1-\partial_XF_2,G\y|.
\end{align}
\end{Lemma}
\begin{proof}
Take the inner product of $(\ref{G-Euq})_1$ and $-G$.

First term
\begin{align}\label{a1}
\begin{aligned}
\z\partial_{XX}[U_s^2G_X],-G\y=&\z\partial_X[U_s^2G_X],G_X\y=\z U_sU_{sX}G_X,G_X\y\\
                                      =&O\big(\|G_X\|^2\big).
\end{aligned}
\end{align}
Second term
\begin{align}\label{a2}
\begin{aligned}
\z\partial_{XY}[U_s^2G_Y],-G\y=&\z\partial_X[U_s^2G_Y],G_Y\y=\z U_sU_{sX}G_Y,G_Y\y\\
                              =&O\big(\|G_Y\|^2\big).
\end{aligned}
\end{align}

Bi-Laplacian term is
\[
\z-\e\Delta^2\Phi,-G\omega\y=\e\z\Phi_{XXXX}+2\Phi_{XXYY}+\Phi_{YYYY},G\omega\y.
\]
\begin{align}\label{a3.1}
\begin{aligned}
\e\z\Phi_{XXXX},G\y=&-\e\z\Phi_{XXX},G_X\y=\e\z\Phi_{XX},G_{XX}\y\\
                         =&\e\z U_sG_{XX}+2U_{sX}G_X+U_{sXX}G,G_{XX}\y\\
                         =&\e\z U_sG_{XX},G_{XX}\y-\e\z2U_{sXX}G_X+U_{sXXX}G,G_X\y,\\
                         =&\e\z U_sG_{XX},G_{XX}\y+O\big(\e\|G_X\|^2\big).
\end{aligned}
\end{align}
Next
\begin{align}\label{a3.2}
\begin{aligned}
\z-2\e\Phi_{XXYY},-G\y=&-\z2\e\Phi_{XXY},G_Y\y=\z2\e\Phi_{XY},G_{XY}\y\\
                      =&2\e\z U_sG_{XY}+U_{sX}G_Y+U_{sY}G_X+U_{sXY}G,G_{XY}\y,\\
                      =&2\e\z U_sG_{XY},G_{XY}\y-\e\z U_{sXX}G_Y,G_Y\y-\e\z U_{sYY}G_X,G_X\y\\
      &-2\e\z U_{sXY}G_Y,G_X\y-2\e\z U_{sXYY}G,G_X\y\\
      =&2\e\z U_sG_{XY},G_{XY}\y+O\big(\e\|G_Y\|^2+\e\|U_{sYY}\|_\infty\|G_X\|^2+\e\|U_{sXYY}\|_\infty\|G_X\|\|G\|\big)\\
=&2\e\z U_sG_{XY},G_{XY}\y+O\big(\|G_X\|^2+\|G_Y\|^2\big).
\end{aligned}
\end{align}
\begin{align}\label{a3.3}
\begin{aligned}
\z-\e\Phi_{YYYY},-G\y=&-\e\z\Phi_{YYY},G_Y\y=\e\z\Phi_{YY},G_Y\y|_{Y=0}+\e\z\Phi_{YY},G_{YY}\y\\
                     =&\e\z U_sG_{YY}+2U_{sY}G_Y+U_{sYY}G,G_Y\y|_{Y=0}\\
                      &+\e\z U_sG_{YY}+2U_{sY}G_Y+U_{sYY}G,G_{YY}\y\\
                     =&2\e\z U_{sY}G_Y,G_Y\y|_{Y=0}+\e\z U_sG_{YY},G_{YY}\y+\e\z2U_{sY}G_Y+U_{sYY}G,G_{YY}\y\\
                     =&\e\z U_{sY}G_Y,G_Y\y|_{Y=0}+\e\z U_sG_{YY},G_{YY}\y-\e\z2U_{sYY}G_Y+U_{sYYY}G,G_Y\y\\
   =&\e\z U_{sY}G_Y,G_Y\y|_{Y=0}+\e\z U_sG_{YY},G_{YY}\y\\
    &+O\big(\e\|U_{sYY}\|_\infty\|G_Y\|^2+\e\|YU_{sYYY}\|_\infty\|\frac{G}{Y}\|\|G_Y\|\big)\\
=&\e\z U_{sY}G_Y,G_Y\y|_{Y=0}+\e\z U_sG_{YY},G_{YY}\y+O\big(\|G_Y\|^2\big),
\end{aligned}
\end{align}
where we use the Hardy inequality $\|\frac{G}{Y}\|\lesssim\|G_Y\|$ and (\ref{main pro}). Since $U_{sY}|_{Y=0}=u^0_{eY}(X,0)+\frac{1}{\se}u^0_{py}(X,0)>0$, the first two terms are positive above. According to (\ref{main pro}),
\[
\begin{split}
 &\|V_{sY}\|_\infty\lesssim\|\p_Y(v^0_e+\se v^0_b)\|_\infty+\se\lesssim\|v^0_{eY}+v^0_{by}\|_{\infty}+\se\lesssim 1,\\
 &\|V_sU_{sY}\|_\infty\lesssim\|v^0_eU_{sY}\|_\infty+1\lesssim\|\frac{v^0_e}{Y}\|_\infty\|YU_{sY}\|_\infty+1\lesssim1,\\
  &\|\Phi_X\|=\|U_sG_X+U_{sX}G\|\lesssim\|G_X\|,\\
&\|\Phi_Y\|=\|U_sG_Y+U_{sY}G\|\lesssim\|G_Y\|+\|U_{sY}Y\|_{\infty}\|\frac{G}{Y}\|\lesssim\|G_Y\|.
\end{split}
\]
The $R[\Phi]$ term can be estimated as
\begin{align}\label{a4}
\begin{aligned}
\z V_s\Phi_{XXY},-G\y=&\z V_s\Phi_{XY},G_X\y+\z V_{sX}\Phi_{XY},G\y\\
                     =&\z V_s(U_sG_{XY}+U_{sX}G_Y+U_{sY}G_X+U_{sXY}G),G_X\y\\
                      &-\z V_{sX}\Phi_X,G_Y\y-\z V_{sXY}\Phi_X,G\y\\
                     =&-\frac{1}{2}\z(V_sU_s)_YG_X,G_X\y+\z V_s(U_{sX}G_Y+U_{sY}G_X+U_{sXY}G),G_X\y\\
                      &-\z V_{sX}\Phi_X,G_Y\y-\z V_{sXY}\Phi_X,G\y\\
                      =&O\big(\|V_sU_{sY}\|_\infty\|G_X\|^2+\|V_sU_{sXY}\|_{\infty}\|G\|\|G_X\|+\|G_X\|^2+\|G_Y\|^2\big)\\
                     =&O\big(\|G_X\|^2+\|G_Y\|^2\big).
\end{aligned}
\end{align}
\begin{align}\label{a5}
\begin{aligned}
\z V_s\Phi_{YYY},-G\y=&\z V_s\Phi_{YY},G_Y\y+\z V_{sY}\Phi_{YY},G\y\\
                     =&\z V_s(U_sG_{YY}+2U_{sY}G_Y+U_{sYY}G),G_Y\y\\
                       &-\z V_{sY}\Phi_Y,G_Y\y-\z V_{sYY}\Phi_Y,G\y\\
                     =&-\frac{1}{2}\z (V_sU_s)_YG_Y,G_Y\y+\z V_s(2U_{sY}G_Y+U_{sYY}G),G_Y\y\\
                      &-\z V_{sY}\Phi_Y,G_Y\omega\y-\z V_{sYY}\Phi_Y,G\y\\
                      =&O\big(\|G_Y\|^2+\|\frac{V_s}{Y}\|_\infty\|Y^2U_{sYY}\|_\infty\|\frac{G}{Y}\|\|G_Y\|+\|YV_{sYY}\|_\infty\|\frac{G}{Y}\|\|\Phi_Y\|\big)\\
                     =&O\big(\|G_Y\|^2\big).
\end{aligned}
\end{align}
\begin{align}\label{a6}
\begin{aligned}
\z -U_{sX}\Delta\Phi,-G\y=&-\z U_{sX}\Phi_X,G_X\y-\z U_{sXX}\Phi_X,G\y\\
               &-\z U_{sX}\Phi_Y,G_Y\y-\z U_{sXY}\Phi_Y,G\y\\
                           =&O\big(\|G_X\|^2+\|G_Y\|^2\big).
\end{aligned}
\end{align}
\begin{align}\label{a7}
\begin{aligned}
\z-\Phi_Y\Delta V_s+\Phi\Delta U_{sX},-G\y=&O\big(\|G_X\|^2+\|G_Y\|^2\big).
\end{aligned}
\end{align}
Collect (\ref{a1})-(\ref{a7}), we can obtain the follow inequality:
\[
\begin{split}
\|G\|_{\Y}^2\lesssim\|G_X\|^2+\|G_Y\|^2+|\z\partial_YF_1-\partial_XF_2,G\y|.
\end{split}
\]
By lemma \ref{Hardy}, for any $0<\xi\leqslant1$,
\[
\begin{split}
\|G_X\|^2\lesssim&\frac{1}{\xi^2}\|U_sG_X\|^2+\xi\e\|\sqrt{U_s}G_{XY}\|^2\\
    \lesssim&\frac{1}{\xi^2}\|G\|^2_{\X}+\xi\|G\|^2_\Y.
\end{split}
\]
Similarly,
\[
\begin{split}
\|G_Y\|^2\lesssim&\frac{1}{\xi^2}\|U_sG_Y\|^2+\xi\e\|\sqrt{U_s}G_{YY}\|^2\\
    \lesssim&\frac{1}{\xi^2}\|G\|^2_{\X}+\xi\|G\|^2_\Y.
\end{split}
\]
So we have
\[
\begin{split}
\|G\|_{\Y}^2\lesssim\frac{1}{\xi^2}\|G\|^2_{\X}+\xi\|G\|^2_\Y+|\z\partial_YF_1-\partial_XF_2,G\y|.
\end{split}
\]
By choosing $\xi$ small enough, we can obtain the inequality (\ref{rou}).
\end{proof}

Above lemma shows the second derivatives of $G$, but it is not good because $\e$ is small. Next lemma shows a critical estimate about derivatives of $G$. 
\begin{Lemma}\label{important1}
Let $G$ be the solution of equation (\ref{G-Euq}), then
\begin{align}
&\frac{3}{2}\z U_s^2G_X,G_X\y+\frac{1}{2}\z U_s^2G_Y,G_Y\y+\z V_sU_sG_X,G_Y\y\\\n
&\lesssim (L+\se)(\|G\|_{\X}^2+\|G\|_{\Y}^2)+|\z\partial_YF_1-\partial_XF_2,G\omega\y|.
\end{align}
\end{Lemma}
\begin{proof}
Take the inner product of $(\ref{G-Euq})_1$ with $-G\omega$, where $\omega=L-x$.

Because $\omega\leqslant L$, the first term is
\begin{align}\label{b1}
\begin{aligned}
\z\partial_{XX}[U_s^2G_X],-G\omega\y=&-\z\partial_X[U_s^2G_X],G\y+\z\partial_X[U_s^2G_X],G_X\omega\y\\
                                    =&\frac{3}{2}\z U_s^2G_X,G_X\y+\z U_sU_{sX}G_X,G_X\omega\y\\
                                    =&\frac{3}{2}\z U_s^2G_X,G_X\y+O\big(L\|G_X\|^2\big).
\end{aligned}
\end{align}
Second term
\begin{align}\label{b2}
\begin{aligned}
\z\partial_{XY}[U_s^2G_Y],-G\omega\y=&\z\partial_X[U_s^2G_Y],G_Y\omega\y\\
                                  =&\z U_s^2G_{XY},G_Y\omega\y+\z2U_sU_{sX}G_Y,G_Y\omega\y\\
                                  =&\frac{1}{2}\z U_s^2G_Y,G_Y\y+\z U_sU_{sX}G_Y,G_Y\omega\y\\
                =&\frac{1}{2}\z U_s^2G_Y,G_Y\y+O\big(L\|G_Y\|^2\big).
\end{aligned}
\end{align}
Bi-Laplacian term is
\[
\z-\e\Delta^2\Phi,-G\omega\y=\e\z\Phi_{XXXX}+2\Phi_{XXYY}+\Phi_{YYYY},G\omega\y.
\]
\begin{align}\label{b3.1}
\begin{aligned}
\e\z\Phi_{XXXX},G\omega\y=&-\e\z\Phi_{XXX},G_X\omega\y+\e\z\Phi_{XXX},G\y\\
                         =&\e\z\Phi_{XX},G_{XX}\omega\y-2\e\z\Phi_{XX},G_X\y\\
                         =&\e\z U_sG_{XX}+2U_{sX}G_X+U_{sXX}G,G_{XX}\omega\y\\
                          &-2\e\z U_sG_{XX}+2U_{sX}G_X+U_{sXX}G,G_X\y\\
                         =&\e\z U_sG_{XX},G_{XX}\omega\y-\e\z 2U_{sX}G_X+U_{sXX}G,G_X\y\\
                          &-\e\z2U_{sXX}G_X+U_{sXXX}G,G_X\omega\y\\
                =&\e\z U_sG_{XX},G_{XX}\omega\y+O\big(\e\|G_X\|\big).
\end{aligned}
\end{align}
Next
\[
\begin{split}
2\e\z\Phi_{XXYY},G\omega\y=&-2\e\z\Phi_{XXY},G_Y\omega\y=2\e\z\Phi_{XY},G_{XY}\omega\y-2\e\z\Phi_{XY},G_Y\y\\
                          =&2\e\z U_sG_{XY}+U_{sX}G_Y+U_{sY}G_X+U_{sXY}G,G_{XY}\omega\y\\
                           &-2\e\z U_sG_{XY}+U_{sX}G_Y+U_{sY}G_X+U_{sXY}G,G_Y\y.
\end{split}
\]
$2\e\z U_sG_{XY},G_{XY}\omega\y$ is good, and
\[
\begin{split}
  &2\e\z U_{sX}G_Y+U_{sY}G_X+U_{sXY}G,G_{XY}\omega\y\\
 =&-\e\z U_{sXX}G_Y,G_Y\omega\y+\e\z U_{sX}G_Y,G_Y\y-\e\z U_{sYY}G_X,G_X\omega\y\\
  &-2\e\z U_{sXY}G_Y,G_X\omega\y-2\e\z U_{sXYY}G,G_Y\omega\y\\
 =&O\big(\e\|G_Y\|^2+\e\|U_{sYY}\|_\infty\|G_X\sqrt\omega\|^2\\
  &+\e\|U_{sXY}\|_\infty\|G_X\|\|G_Y\|+\e\|YU_{sXYY}\|_\infty\|\frac{G}{Y}\|\|G\|\big)\\
 =&O\big((L+\se)(\|G_X\|^2+\|G_Y\|^2)\big),
\end{split}
\]

and
\[
\begin{split}
&-2\e\z U_sG_{XY}+U_{sX}G_Y+U_{sY}G_X+U_{sXY}G,G_Y\y\\
=&-\e\z U_{sX}G_Y,G_Y\y-2\e\z U_{sY}G_Y,G_X\y-2\e\z U_{sXY}G,G_Y\y\\
=&O\big(\e\|G_Y\|^2+\e\|U_{sY}\|_\infty\|G_X\|\|G_Y\|+\e\|U_{sXY}Y\|_\infty\|\frac{G}{Y}\|\|G_Y\|\big)\\
=&O\big((L+\se)(\|G_X\|^2+\|G_Y\|^2)\big).
\end{split}
\]
Therefore
\begin{align}\label{b3.2}
\begin{aligned}
2\e\z\Phi_{XXYY},G\omega\y=&2\e\z U_sG_{XY},G_{XY}\omega\y+O\big((L+\se)\|\nabla G\|^2\big).
\end{aligned}
\end{align}
Integrating by parts, we have
\[
\begin{split}
\e\z\Phi_{YYYY},G\omega\y=&-\e\z\Phi_{YYY},G_Y\omega\y=\e\z \Phi_{YY},G_Y\omega\y|_{Y=0}+\e\z\Phi_{YY},G_{YY}\omega\y\\
                         =&\e\z U_sG_{YY}+2U_{sY}G_Y+U_{sYY}G,G_Y\omega\y|_{Y=0}\\
                          &+\e\z U_sG_{YY}+2U_{sY}G_Y+U_{sYY}G,G_{YY}\omega\y\\
                         =&2\e\z U_{sY}G_Y,G_Y\omega\y|_{Y=0}+\e\z U_sG_{YY},G_{YY}\omega\y+\e\z2U_{sY}G_Y+U_{sYY}G,G_{YY}\omega\y\\
                         =&\e\z U_{sY}G_Y,G_Y\omega\y|_{Y=0}+\e\z U_sG_{YY},G_{YY}\omega\y-\e\z 2U_{sYY}G_Y+U_{sYYY}G,G_Y\omega\y.
\end{split}
\]
Because $U_{sY}|_{Y=0}>0$, the first two terms are positive above, and
\[
\begin{split}
-\e\z 2U_{sYY}G_Y+U_{sYYY}G,G_Y\omega\y=&O\big(\e L\|U_{sYY}\|_\infty\|G_Y\|^2+\e L\|U_{sYYY}Y\|_\infty\|\frac{G}{Y}\|\|G_Y\|\big)\\
=&O\big(L\|G_Y\|^2\big),
\end{split}
\]
then we obtain
\begin{align}\label{b3.3}
\begin{aligned}
\e\z\Phi_{YYYY},G\omega\y=&
\e\z U_{sY}G_Y,G_Y\omega\y|_{Y=0}+\e\z U_sG_{YY},G_{YY}\omega\y+O\big(L\|G_Y\|^2\big).
\end{aligned}
\end{align}
Finally, we deal with $R[\Phi]$ term. Since
\[
\begin{split}
\z V_s\Phi_{XXY},-G\omega\y=&\z V_s\Phi_{XY},G_X\omega\y+\z V_{sX}\Phi_{XY},G\omega\y-\z V_s\Phi_{XY},G\y\\
                           =&\z V_s(U_sG_{XY}+U_{sX}G_Y+U_{sY}G_X+U_{sXY}G),G_X\omega\y\\
                            &-\z V_{sX}\Phi_X,G_Y\omega\y-\z V_{sXY}\Phi_X,G\omega\y+\z V_s\Phi_X,G_Y\y+\z V_{sY}\Phi_X,G\y\\
                           =&-\frac{1}{2}\z(V_sU_s)_YG_X,G_X\omega\y+\z V_s(U_{sX}G_Y+U_{sY}G_X+U_{sXY}G),G_X\omega\y\\
                            &-\z V_{sX}\Phi_X,G_Y\omega\y-\z V_{sXY}\Phi_X,G\omega\y+\z V_s\Phi_X,G_Y\y+\z V_{sY}\Phi_X,G\y.
\end{split}
\]
Notice that
\[
\begin{split}
&\|V_sU_{sY}\|_{\infty}\lesssim\frac{\|v^0_eu^0_{by}\|_{\infty}}{\se}+1\lesssim\|\frac{v^0_e}{Y}\|_{\infty}\|yu^0_{by}\|_{\infty}+1\lesssim1,\\
&\|\Phi_X\|=\|U_sG_X+U_{sX}G\|\lesssim\|G_X\|,\\
&\|\Phi_Y\|=\|U_sG_Y+U_{sY}G\|\lesssim\|G_Y\|+\|U_{sY}Y\|_{\infty}\|\frac{G}{Y}\|\lesssim\|G_Y\|,\\
&\|G\|\lesssim L\|G_X\|.
\end{split}
\]
We obtain
\begin{align}\label{b4}
\begin{aligned}
\z V_s\Phi_{XXY},-G\omega\y=&\z V_s\Phi_X,G_Y\y+O\big(L\|V_sU_{sY}\|_\infty\|G_X\|^2+L\|G_X\|^2+L\|G_Y\|^2+L\|G\|^2\big)\\
                           =&\z V_sU_sG_X,G_Y\y+\z V_sU_{sX}G,G_Y\y+O\big(L(\|G_X\|^2+\|G_Y\|^2)\big)\\
                           =&\z V_sU_sG_X,G_Y\y+O\big(L(\|G_X\|^2+\|G_Y\|^2)\big).
\end{aligned}
\end{align}
The others are easy. In fact,
\begin{align}\label{b5}
\begin{aligned}
\z V_s\Phi_{YYY},-G\omega\y=&\z V_s\Phi_{YY},G_Y\omega\y+\z V_{sY}\Phi_{YY},G\omega\y\\
                           =&\z V_s(U_sG_{YY}+2U_{sY}G_Y+U_{sYY}G),G_Y\omega\y\\
                            &-\z V_{sY}\Phi_Y,G_Y\omega\y-\z V_{sYY}\Phi_Y,G\omega\y\\
                           =&-\frac{1}{2}\z (V_sU_s)_YG_Y,G_Y\omega\y+\z V_s(2U_{sY}G_Y+U_{sYY}G),G_Y\omega\y\\
                            &-\z V_{sY}\Phi_Y,G_Y\omega\y-\z V_{sYY}\Phi_Y,G\omega\y\\
                           =&O\big(L\|G_Y\|^2+L\|\frac{V_s}{Y}\|_\infty\|Y^2U_{sY}\|_\infty\|\frac{G}{Y}\|\|G_Y\|+L\|YV_{sYY}\|_\infty\|\frac{G}{Y}\|\|\Phi_Y\|\big)\\
                           =&O\big(L\|G_Y\|^2\big)
\end{aligned}
\end{align}
\begin{align}\label{b6}
\begin{aligned}
\z -U_{sX}\Delta\Phi,-G\omega\y=&-\z U_{sX}\Phi_X,G_X\omega\y-\z U_{sXX}\Phi_X,G\omega\y+\z U_{sX}\Phi_X,G\y\\
                                &-\z U_{sX}\Phi_Y,G_Y\omega\y-\z U_{sXY}\Phi_Y,G\omega\y\\
                                =&O\big(L(\|G_X\|^2+\|G_Y\|^2)\big).
\end{aligned}
\end{align}
\begin{align}\label{b7}
\begin{aligned}
\z-\Phi_Y\Delta V_s+\Phi\Delta U_{sX},-G\omega\y=&O\big(L(\|G_X\|^2+\|G_Y\|^2)\big).
\end{aligned}
\end{align}
Collect (\ref{b1})-(\ref{b7}), we get
\[
\begin{split}
&\frac{3}{2}\z U_s^2G_X,G_X\y+\frac{1}{2}\z U_s^2G_Y,G_Y\y+\z V_sU_sG_X,G_Y\y\\
&\lesssim (L+\se)\|\nabla G\|^2+|\z\partial_YF_1-\partial_XF_2,G\omega\y|.
\end{split}
\]
By Lemma \ref{Hardy} for $\xi=1$,
\[
\begin{split}
\|\nabla G\|^2\lesssim\|G\|_{\X}^2+\|G\|_{\Y}^2.
\end{split}
\]
Finally we have 
\[
\begin{split}
&\frac{3}{2}\z U_s^2G_X,G_X\y+\frac{1}{2}\z U_s^2G_Y,G_Y\y+\z V_sU_sG_X,G_Y\y\\
\lesssim &(L+\se)(\|G\|_{\X}^2+\|G\|_{\Y}^2)+|\z\partial_YF_1-\partial_XF_2,G\omega\y|.
\end{split}
\]
So we finish the proof.
\end{proof}

Since $v^0_e\neq0$, we need to deal with $\z V_sU_sG_X,G_Y\y$ in above lemma. Notice $V_s\approx v^0_e$ is small near the boundary $\{(X,Y)|Y=0\}$, it is sufficient to estimate the stream-function away from the boundary layer. From now on, we write $\ed(Y):=\eta(\frac{Y}{\delta})$. Here $\eta$ is a smooth non-negative function s.t. $\eta|_{[0,1]}=0$ and $\eta|_{[2,\infty)}=1$. $\delta$ is a small constant satisfying $0<L^\frac{1}{2}+\e^\frac{1}{4}\leqslant\delta\ll1$. Next we show the key estimate of $\Phi_X+\frac{V_s}{U_s}\Phi_Y$ in the domain $\{(X,Y)|Y\geqslant 2\delta)\}$.
\begin{Lemma}\label{non-shear}
Let $G$ be the solution of equation (\ref{Phi-Euq}), $0<L^\frac{1}{2}+\e^\frac{1}{4}\leqslant\delta\ll1$, then
\begin{align}
\z \Phi_X+\frac{V_s}{U_s}\Phi_Y,(\Phi_X+\frac{V_s}{U_s}\Phi_Y)\ed\y\lesssim&(L+\se)(\|G\|^2_{\X}+\|G\|^2_{\Y})+\|F_1\|^2+\|F_2\|^2
\end{align}
\end{Lemma}
\begin{proof}
Take the inner product of $(\ref{Phi-Euq})_1$ with $-\frac{1}{U_s}[\Phi_X+Q_s\Phi_Y]\ot(X)\ed(Y)$, where $Q_s=\frac{V_s}{U_s}$, $\ot=x(L-x)$.

Because 
\[
\begin{split}
U_s\Delta\Phi_X+V_s\Delta\Phi_Y&=U_s[\Delta\Phi_X+Q_s\Delta\Phi_Y]\\
                               &=U_s\Delta[\Phi_X+Q_s\Phi_Y]-U_s[2\nabla Q_s\cdot\nabla\Phi_Y+(\Delta Q_s) \Phi_Y],
\end{split}
\]

We have 
\[
\begin{split}
&\z U_s\Delta\Phi_X+V_s\Delta\Phi_Y, -\frac{1}{U_s}[\Phi_X+Q_s\Phi_Y]\ot(X)\ed(Y)\y\\
=&-\z \Delta[\Phi_X+Q_s\Phi_Y],(\Phi_X+Q_s\Phi_Y)\ot\ed\y\\
 &+\z 2\nabla Q_s\cdot\nabla\Phi_Y+(\Delta Q_s) \Phi_Y, (\Phi_X+Q_s\Phi_Y)\ot\ed\y.
\end{split}
\]
We calculate the first part in the right hand of above equality:
\[
\begin{split}
-\z \Delta[\Phi_X+Q_s\Phi_Y],(\Phi_X+Q_s\Phi_Y)\ot\ed\y=&\|\nabla[\Phi_X+Q_s\Phi_Y]\sqrt{\ot\ed}\|^2\\
   &+\z[\Phi_X+Q_s\Phi_Y]_X,[\Phi_X+Q_s\Phi_Y]\ot_X\ed\y\\
   &+\z[\Phi_X+Q_s\Phi_Y]_Y,[\Phi_X+Q_s\Phi_Y]\ot\ed'\y\\
  =&\|\nabla[\Phi_X+Q_s\Phi_Y]\sqrt{\ot\ed}\|^2+\|[\Phi_X+Q_s\Phi_Y]\sqrt{\ed}\|^2\\
   &+\frac{1}{2}\z\Phi_X+Q_s\Phi_Y,[\Phi_X+Q_s\Phi_Y]\ot\ed''\y.\\ 
\end{split}
\]
Because $\ot\lesssim L^2$, $|\ed''|\lesssim\frac{1}{\delta^2}\lesssim \frac{1}{L}$ and $Q_s=O\big(1\big)$, 
\[
\begin{split}
\z\Phi_X+Q_s\Phi_Y,[\Phi_X+Q_s\Phi_Y]\ot\ed''\y\lesssim L\|\na \Phi\|^2.
\end{split}
\]

According to the fact (\ref{away b}), when $y\geqslant\delta$,
\begin{align}\label{fact}
\begin{aligned}
\na^j U_s&=\na^j u^0_e+O\big(\se\big)=O\big(1\big),\\
\na^j V_s&=\na^j v^0_e+O\big(\se\big)=O\big(1\big),\\
\na^j Q_s&=\na^j\big[\frac{v^0_e}{u^0_e}\big]+O\big(\se\big)=O\big(1\big), 
\end{aligned}
\end{align}
the second part can be estimated by the following way:  
\[
\begin{split}
&\z 2\nabla Q_s\cdot\nabla\Phi_Y+(\Delta Q_s) \Phi_Y, (\Phi_X+Q_s\Phi_Y)\ot\ed\y\\
= & 2\z Q_{sX}\Phi_{XY}, \Phi_X\ot\ed\y+2\z Q_{sY}\Phi_{YY},\Phi_X\ot\ed\y\\
   &+  2\z Q_{sX}\Phi_{XY}, Q_s\Phi_Y\ot\ed\y+2\z Q_{sY}\Phi_{YY},Q_s\Phi_Y\ot\ed\y\\
   &+\z(\Delta Q_s) \Phi_Y, (\Phi_X+Q_s\Phi_Y)\ot\ed\y\\
= & -\z [Q_{sX}\ot\ed]_Y,\Phi_X^2\y+\z [Q_{sY}\ot\ed]_X,\Phi_{Y}^2\y\\
  & -2\z [Q_{sY}\ot\ed]_Y,\Phi_Y\Phi_X\y-\z [Q_{sX}\ot\ed]_X,\Phi_Y^2\y\\
  & -\z [Q_{sY}Q_s\ot\ed]_Y,\Phi_Y^2\y+\z(\Delta Q_s) \Phi_Y, (\Phi_X+Q_s\Phi_Y)\ot\ed\y\y\\
= &O\big(L\|\Phi_X\|^2+L\|\Phi_Y\|^2\big).
\end{split}
\]
Here we use the fact $\ot\lesssim L^2$, $|\ot_X|\lesssim L$ and $|\ed'|\lesssim\de\lesssim\frac{1}{\sqrt{L}}$. So we conclude that
\begin{align}\label{c1}
\begin{aligned}
&\z U_s\Delta\Phi_X+V_s\Delta\Phi_Y, -\frac{1}{U_s}[\Phi_X+Q_s\Phi_Y]\ot(X)\ed(Y)\y\\
=&\|\nabla[\Phi_X+Q_s\Phi_Y]\sqrt{\ot\ed}\|^2+\|[\Phi_X+Q_s\Phi_Y]\sqrt{\ed}\|^2+O\big(L\|\na \Phi\|^2\big).
\end{aligned}
\end{align}

Next we deal with the bi-Laplacian term. 
\[
\begin{split}
\e\z\Delta^2\Phi, \frac{1}{U_s}[\Phi_X+Q_s\Phi_Y]\ot\ed\y=
\e\z\Phi_{XXXX}+2\Phi_{XXYY}+\Phi_{YYYY},\frac{1}{U_s}[\Phi_X+Q_s\Phi_Y]\ot\ed\y.
\end{split}
\]
Integration by Parts is allowed in $X$ direction because $\Phi_X|_{\p \Omega}=0,\Phi_{\p \Omega}=0$ and $\ot|_{X=0,X=L}=0$, in $Y$ direction because $\eta|_{[0,1]}=0$. Use the inequality (\ref{fact}) again, we can see 
\[
\begin{split}
\e\z\Phi_{XXXX},\Phi_X\frac{\ot\ed}{U_s}\y=&-\e\z\Phi_{XXX},\Phi_{XX}\frac{\ot\ed}{U_s}\y-\e\z\Phi_{XXX},\Phi_X\p_X[\frac{\ot\ed}{U_s}]\y\\
=&\frac{3\e}{2}\z\Phi_{XX}^2,\p_X[\frac{\ot\ed}{U_s}]\y+\e\z\Phi_{XX},\Phi_X\p^2_X[\frac{\ot\ed}{U_s}]\y\\
=&O\big(L\e\|\Phi_{XX}\|^2+\e\|\Phi_X\|\|\Phi_{XX}\|\big)\\
=&O\big(\se\|\Phi_X\|^2+(L+\se)\e\|\Phi_{XX}\|^2\big).
\end{split}
\]
\[
\begin{split}
\e\z\Phi_{XXXX},\Phi_Y\frac{Q_s\ot\ed}{U_s}\y=&-\e\z\Phi_{XXX},\Phi_{XY}\frac{Q_s\ot\ed}{U_s}\y-\e\z\Phi_{XXX},\Phi_Y\p_X[\frac{Q_s\ot\ed}{U_s}]\y\\
=&\e\z\Phi_{XX},\Phi_{XXY}\frac{Q_s\ot\ed}{U_s}\y+2\e\z\Phi_{XX},\Phi_{XY}\p_X[\frac{Q_s\ot\ed}{U_s}]\y+\e\z\Phi_{XX},\Phi_{Y}\p^2_{X}[\frac{Q_s\ot\ed}{U_s}]\y\\
=&-\frac{\e}{2}\z\Phi_{XX}^2,\p_Y[\frac{Q_s\ot\ed}{U_s}]\y+O\big(L\e\|\nabla^2\Phi\|^2+\e\|\Phi_{Y}\|\|\Phi_{XX}\|\big)\\
=&O\big(\se\|\Phi_Y\|^2+(L+\se)\e\|\na^2\Phi\|^2\big).
\end{split}
\]
\[
\begin{split}
2\e\z\Phi_{XXYY},\Phi_X\frac{\ot\ed}{U_s}\y=&-2\e\z\Phi_{XXY},\Phi_{XY}\frac{\ot\ed}{U_s}\y-2\e\z\Phi_{XXY},\Phi_X\p_Y[\frac{\ot\ed}{U_s}]\y\\
=&\e\z\Phi^2_{XY},\p_X[\frac{\ot\ed}{U_s}]\y+2\e\z\Phi_{XY},\Phi_{XX}\p_Y[\frac{\ot\ed}{U_s}]\y+2\e\z\Phi_{XY},\Phi_X\p^2_{XY}[\frac{\ot\ed}{U_s}]\y\\
=&O\big(L\e\|\na^2\Phi\|^2+\e\|\Phi_X\|\|\Phi_{XY}\|\big)\\
=&O\big(\sqrt{\e}\|\Phi_X\|^2+(L+\sqrt{\e})\e\|\na^2\Phi\|^2\big).
\end{split}
\]
\[
\begin{split}
2\e\z\Phi_{XXYY},\Phi_Y\frac{Q_s\ot\ed}{U_s}\y=&-2\e\z\Phi_{XYY},\Phi_{XY}\frac{Q_s\ot\ed}{U_s}\y-2\e\z\Phi_{XYY},\Phi_Y\p_X[\frac{Q_s\ot\ed}{U_s}]\y\\
=&\e\z\Phi^2_{XY},\p_Y[\frac{Q_s\ot\ed}{U_s}]\y+2\e\z\Phi_{YY},\Phi_{XY}\p_X[\frac{Q_s\ot\ed}{U_s}]\y+2\e\z\Phi_{YY},\Phi_Y\p^2_X[\frac{Q_s\ot\ed}{U_s}]\y\\
=&O\big(L\e\|\na^2\Phi\|^2+\e\|\Phi_Y\|\|\Phi_{YY}\|\big)\\
=&O\big(\sqrt{\e}\|\Phi_Y\|^2+(L+\sqrt{\e})\e\|\na^2\Phi\|^2\big).
\end{split}
\]
\[
\begin{split}
\e\z\Phi_{YYYY},\Phi_X\frac{\ot\ed}{U_s}\y=&-\e\z\Phi_{YYY},\Phi_{XY}\frac{\ot\ed}{U_s}\y-\e\z\Phi_{YYY},\Phi_X\p_Y[\frac{\ot\ed}{U_s}]\y\\
=&\e\z\Phi_{YY},\Phi_{XYY}\frac{\ot\ed}{U_s}\y+2\e\z\Phi_{YY},\Phi_{XY}\p_Y[\frac{\ot\ed}{U_s}]\y+\e\z\Phi_{YY},\Phi_{X}\p^2_Y[\frac{\ot\ed}{U_s}]\y\\
=&-\frac{\e}{2}\z \Phi_{YY},\Phi_{YY}\p_X[\frac{\ot\ed}{U_s}]\y+O\big(L\e\|\na^2\Phi\|^2+\e\|\Phi_X\|\|\Phi_{YY}\|\big)\\
=&O\big(\sqrt{\e}\|\Phi_X\|^2+(L+\sqrt{\e})\e\|\na^2\Phi\|^2\big).
\end{split}
\]
\[
\begin{split}
\e\z\Phi_{YYYY},\Phi_Y\frac{Q_s\ot\ed}{U_s}\y=&-\e\z\Phi_{YYY},\Phi_{YY}\frac{Q_s\ot\ed}{U_s}\y-\e\z\Phi_{YYY},\Phi_Y\p_Y[\frac{Q_s\ot\ed}{U_s}]\y\\
=&\frac{3\e}{2}\z\Phi^2_{YY},\p_Y[\frac{Q_s\ot\ed}{U_s}]\y+\e\z\Phi_{YY},\Phi_Y\p^2_Y[\frac{Q_s\ot\ed}{U_s}]\y\\
=&O\big(L\e\|\na^2\Phi\|^2+\e\|\Phi_Y\|\|\Phi_{YY}\|\big)\\
=&O\big(\sqrt{\e}\|\Phi_Y\|^2+(L+\sqrt{\e})\e\|\na^2\Phi\|^2\big).
\end{split}
\]
So we conclude that
\begin{align}\label{c2}
\begin{aligned}
&\e\z\Phi_{XXXX}+2\Phi_{XXYY}+\Phi_{YYYY},\frac{1}{U_s}[\Phi_X+Q_s\Phi_Y]\ot\ed\y\\
=&O\big(\sqrt{\e}\|\na\Phi\|^2+(L+\sqrt{\e})\e\|\na^2\Phi\|^2\big).
\end{aligned}
\end{align}
The other terms are easy. We have
\begin{align}\label{c3}
\begin{aligned}
&\z\Phi_X\Delta U_s+\Phi_Y\Delta V_s,\frac{1}{U_s}[\Phi_X+Q_s\Phi_Y]\ot\ed\y=O\big(L^2\|\na\Phi\|^2\big).
\end{aligned}
\end{align}
Collect (\ref{b1})-(\ref{b7}), we obtain
\[
\begin{split}
&\|\nabla[\Phi_X+Q_s\Phi_Y]\sqrt{\ot\ed}\|^2+\|[\Phi_X+Q_s\Phi_Y]\sqrt{\ed}\|^2\\
\lesssim&(L+\se)(\|\na\Phi\|^2+\e\|\na^2\Phi\|^2)+|\z\p_XF_1-\p_YF_2,\frac{1}{U_s}[\Phi_X+Q_s\Phi_Y]\ot\ed\y|\\
\lesssim&(L+\se)(\|G\|^2_{\X}+\|G\|^2_{\Y})+\|F_1\|^2+\|F_2\|^2
\end{split}
\]

So we end the proof.
\end{proof}

{\bf Proof of Proposition \ref{Prop}:}\\
Notice that when $Y\leqslant \se$, $V_s\lesssim\se U_s\lesssim\de U_s$, and when $\se\leqslant Y\leqslant2\delta$, $V_s\lesssim \de U_s$. So 
\[
\begin{split}
\z V_sU_sG_X,G_Y\y=&\z U_sG_X,V_sG_Y\ed\y+\z V_sU_sG_X,G_Y(1-\ed)\y\\
=&-\z U_s^2G_X,G_X\ed\y+\z U_s^2G_X,(G_X+\frac{V_s}{U_s}G_Y)\ed\y+O\big(\de \|U_sG_X\|^2+\de\|U_sG_Y\|^2\big).
\end{split}
\]
According to Lemma \ref{non-shear},
\[
\begin{split}
\z U_s^2G_X,(G_X+\frac{V_s}{U_s}G_Y)\ed\y=&\z \Phi_X-U_{sX}G,(U_sG_X+V_sG_Y)\ed\y\\
=&\z \Phi_X,(U_sG_X+V_sG_Y)\ed\y+O\big(\|G\|\|\na G\|\big)\\
=&\z \Phi_X,(\Phi_X+\frac{V_s}{U_s}\Phi_Y)\ed\y-\z\Phi_X,(U_{sX}G+\frac{V_s}{U_s}U_{sY}G)\ed\y+O\big(L\|\na G\|^2\big)\\
\geqslant &-\frac{1}{4}\|\Phi_X\|^2-\|(\Phi_X+\frac{V_s}{U_s}\Phi_Y)\ed\|^2+O\big(L\|\na G\|^2\big)\\
\geqslant &-\frac{1}{4}\|U_sG_X\|^2+O\big((L+\se)(\|G\|^2_{\X}+\|G\|^2_{\Y})+(\|F_1\|^2+\|F_2\|^2)\big).
\end{split}
\]
So, we obtain
\[
\begin{split}
\z V_sU_sG_X,G_Y\y=&\z U_sG_X,V_sG_Y\ed\y+\z V_sU_sG_X,G_Y(1-\ed)\y\\
\geqslant&-\frac{5}{4}\z U_s^2G_X,G_X\ed\y+O\big(\de \|U_sG_X\|^2+\de\|U_sG_Y\|^2\big)\\
 &+O\big((L+\se)(\|G\|^2_{\X}+\|G\|^2_{\Y})+(\|F_1\|^2+\|F_2\|^2)\big).
\end{split}
\]
Combine above inequality and Lemma \ref{important1}, we have
\[
\begin{split}
&\frac{1}{4}\z U_s^2G_X,G_X\y+\frac{1}{2}\z U_s^2G_Y,G_Y\y+O\big(\de \|U_sG_X\|^2+\de\|U_sG_Y\|^2\big)\\
&\lesssim (L+\se)(\|G\|_{\X}^2+\|G\|_{\Y}^2)+|\z\partial_YF_1-\partial_XF_2,G\omega\y|+(\|F_1\|^2+\|F_2\|^2).
\end{split}
\]
We can select $\de\geqslant L^\frac{1}{2}+\e^\frac{1}{4}$ is a small enough, then, 
\[
\begin{split}
\|G\|_\X^2&\lesssim(L+\se)\|G\|_\Y^2+|\z\partial_YF_1-\partial_XF_2,G\omega\y|+(\|F_1\|^2+\|F_2\|^2).
\end{split}
\]
By Lemma \ref{routine},
\[
\begin{split}
\|G\|_\Y^2\lesssim\|G\|_\X^2+|\z\partial_YF_1-\partial_XF_2,G\y|.
\end{split}
\]
Combine the above two equalities, because $L+\se$ is small,
\[
\begin{split}
\|G\|_\X^2+\|G\|_\Y^2\lesssim&\|G\|_\X^2+|\z\partial_YF_1-\partial_XF_2,G\y|\\
                                    \lesssim&(L+\se)\|G\|_\Y^2+\z\partial_YF_1-\partial_XF_2,G\omega\y|\\
                                    &+|\z\partial_YF_1-\partial_XF_2,G\y|+\|F_1\|^2+\|F_2\|^2\\
                                    \lesssim&(\|F_1\|+\|F_2\|)(\|G_Y\|+\|G_X\|)+\|F_1\|^2+\|F_2\|^2\\
                                    \lesssim&(\|F_1\|+\|F_2\|)(\|G\|_\X+\|G\|_\Y)+\|F_1\|^2+\|F_2\|^2.
\end{split}
\]
It is easy to see
\[
\|\se\Phi_{XX},\se\Phi_{XY},\se\Phi_{YY},\Phi_X,\Phi_Y\|\lesssim\|G\|_\X+\|G\|_\Y\lesssim\|F_1\|+\|F_2\|.
\]
Then we obtain the proof.
\qed

\section{Proof of the main Theorems}
{\bf Proof of Theorem \ref{main}:}
Let $\Us=[U_s,V_s]$, $\U=[U,V]$ and $\RE=[R_1,R_2]$, now we write Navier-Stokes equation in the following form
\begin{equation}\label{U-NS}
\left\{
\begin{aligned}
&-\e\Delta \U+\Us\cdot\nabla\U+\U\cdot\nabla\Us+\U\cdot\nabla \U+\nabla P=-\RE, \\
&\nabla \cdot \U=0,\hspace{5mm} \U|_\Omega=0.
\end{aligned}
\right.
\end{equation}
We use the method of contraction mapping. Define
\[
\begin{split}
\|\U\|_\Z:=\|\U\|+\se\|\nabla\U\|+\e^\frac{3}{2}\|\nabla^2\U\|.
\end{split}
\]
We denote $\T:W^{2,2}(\Omega)\rightarrow W^{2,2}(\Omega)$ as this way, $\T(\U)=\W$ where $\W$ is given by
\begin{equation}
\left\{
\begin{aligned}
&-\e\Delta \W+\Us\cdot\nabla\W+\W\cdot\nabla\Us+\nabla P=-\RE-\U\cdot\nabla\U, \\
&\nabla \cdot \W=0,\hspace{5mm} \W|_\Omega=0.
\end{aligned}
\right.
\end{equation}
Let $B=\{\U\in W^{2,2}(\Omega):\|\U\|_\Z\leqslant C_0\e^{\frac{3}{2}}\}$, $C_0$ is chosen latter. Next we prove $\T$ is a contractive mapping in $B$, if $\|\RE\|\leqslant C_1\e^{\frac{3}{2}}$. We write $\F=-\RE-\U\cdot\nabla \U$, from Proposition \ref{Prop},
\[
\begin{split}
\|\W\|+\se\|\nabla\W\|\lesssim\|\F\|.
\end{split}
\]
Due to the $W^{2,2}$ estimate of Stokes equations in convex polygon in \cite{KE},
\[
\begin{split}
\e\|\nabla^2\W\|\lesssim\|\F\|+\|\nabla\W\|+\frac{1}{\se}\|\W\|\lesssim\frac{1}{\se}\|\F\|.
\end{split}
\]
So we get
\[
\begin{split}
\|\W\|_\Z\leqslant C_2\|\F\|.
\end{split}
\]
It's easy to see
\[
\begin{split}
\|\U\cdot\nabla \U\|\lesssim\|\U\|_{L^\infty}\|\nabla\U\|\lesssim\|\U\|^{\frac{1}{4}}\|\nabla\U\|^{\frac{3}{2}}\|\nabla^2\U\|^{\frac{1}{4}}\lesssim\e^{-\frac{9}{8}}\|\U\|^2_\Z.
\end{split}
\]
It implies
\[
\begin{split}
\|\W\|_\Z\leqslant C_2\|\F\|+C_3\e^{-\frac{9}{8}}\|\U\|^2_\Z\leqslant (C_1C_2+C_3C^2_0\e^\frac{3}{8})\e^{\frac{3}{2}}.
\end{split}
\]
Select $C_0=C_1C_2+1$, $\T(B)\subset B$ when $\e$ is small enough. And if $\U_1, \U_2\in B$,
\[
\begin{split}
\|\T(\U_1-\U_2)\|_\Z&\leqslant C_2\|\U_1\cdot\nabla\U_1-\U_2\cdot\nabla\U_2\|\\
                    &\leqslant C_2\|(\U_1-\U_2)\cdot\nabla\U_1\|+\|\U_2\cdot\nabla(\U_1-\U_2)\|\\
                    &\leqslant C_2\|(\U_1-\U_2)\|_\infty\|\nabla\U_1\|+\|\U_2\|_\infty\|\nabla(\U_1-\U_2)\|\\
                    &\leqslant C_3\e^{-\frac{9}{8}}(\|\U_1\|_\Z+\|\U_2\|_\Z)\|\U_1-\U_2\|_\Z\\
                    &\leqslant 2C_0C_3\e^\frac{3}{8}\|\U_1-\U_2\|_\Z,
\end{split}
\]
so $\T$ is a contraction mapping on $B$ when $\e$ is small enough, we can conclude equations (\ref{U-NS}) admits a solution and
\[
\begin{split}
\|\U\|_{L^\infty}\lesssim\e^{-\frac{5}{8}}\|\U\|_\Z\lesssim\e^{\frac{7}{8}}.
\end{split}
\]
So we have
\[
\begin{split}
 &|U^\e(X,Y)-u^0_e(X,Y)-u^0_b(X,\frac{Y}{\se})|\\
=&|\se u^1_e(X,Y)+\se u^1_b(X,\frac{Y}{\se})+\e u^2_e(X,Y)+\e\hat u^2_b(X,\frac{Y}{\se})+U(X,Y)|\\
\lesssim&\se,\\
 &|V^\e(X,Y)-v^0_e(X,Y)|\\
=&|\se v^0_b(X,\frac{Y}{\se})+\se v^1_e(X,Y)+\e v^1_b(X,\frac{Y}{\se})+\e v^2_e(X,Y)+\e^\frac{3}{2}\hat{v}^2_b(X,\frac{Y}{\se})+V(X,Y)|\\
\lesssim&\se,
\end{split}
\]
which ends the proof.
\qed

\part*{Appendix}
The high-order approximate solutions are constructed similarly in \cite{GZ} even though the leading Euler flow is non-shear. By the method of asymptotic matching expansions, we can deduce the equations of $[u^j_e,v^j_e]$ and $[u^j_b,v^j_b]$, $j=1,2.$
The first order Euler profile $[u^1_e,v^1_e,p^1_e]$ solves the linearized Euler equations around $[u^0_e,v^0_e]$:
\begin{align} \label{euler1}
\left\{
\begin{aligned}
&u^0_eu^1_{eX}+u^0_{eX}u^1_e+v^0_eu^1_{eY}+u^0_{eY}v^1_e+p^1_{eX}=0,\\
&u^0_ev^1_{eX}+v^0_{eX}u^1_e+v^0_ev^1_{eY}+v^0_{eY}v^1_e+p^1_{eY}=0,\\
&\p_X u^1_e + \p_Y v^1_e = 0,\\
&v^1_e|_{Y=0}=-v^0_b|_{y=0}. \\
\end{aligned}
\right.
\end{align}
We eliminate the pressure $p^1_e$, 
\begin{align} \label{stream1}
\left\{
\begin{aligned}
&v^0_e\Delta u^1_{e}-u^0_e\Delta v^1_e+v^1_e\Delta u^0_{e}-u^1_e\Delta v^0_e=0,\\
&\p_X u^1_e + \p_Y v^1_e = 0,\\
&v^1_e|_{Y=0}=-v^0_b|_{y=0}. \\
\end{aligned}
\right.
\end{align}
We introduce independent variables by
\begin{align}\label{vari change}
\begin{aligned}
\theta(X,Y)=X, \hspace{5 mm} \psi(X,Y)=\int_0^Yu^0_e(X,Y')\dd Y'.
\end{aligned}
\end{align}
Let $\psi^1$ be the stream function of $[u^1_e,v^1_e]$
\[
\begin{split}
     \psi^1(X,Y):=\int_0^Yu^1_e(X,Y')\dd Y'-\int_0^Xv^1_e(X',0)\dd X',\hspace{5 mm} \psi^1_Y=u^1_e, \hspace{5 mm} \psi^1_X=-v^1_e,
\end{split}
\]
then first Euler layer equations (\ref{euler1}) are equivalent to
\begin{align} \label{stream-euler1}
\begin{aligned}
     \p_\theta[\Delta_{XY} \psi^1-F_e'(\psi)\psi^1]=0,
\end{aligned}
\end{align}
where $F_e$ is a smooth function in (\ref{F-Euler}). We try to find a solution of the following equations
\begin{align} \label{stream-euler1-bry}
\left\{
\begin{aligned}
     &\Delta \psi^1-F_e'(\psi)\psi^1=0,\\
     &\psi^1|_{X=0}=\psi^1_0(Y),\hspace{3 mm}\psi^1|_{X=L}=\psi^1_L(Y),\\
     &\psi^1|_{Y=0}=\int_0^Xv^0_b(X',0)\dd X',\hspace{3 mm}\psi^1|_{Y\rightarrow\infty}=0.
\end{aligned}
\right.
\end{align}
It is a standard elliptic problem, we have the following result.
\begin{lemma}\label{LEuler}
If $v^0_b$ is a smooth functions, for any $L>0$, if $\psi^1_0(Y)$, $\psi^1_L(Y)$ satisfy the 
compatibility conditions on the corner, then the equations (\ref{stream-euler1-bry}) admit a unique solution satisfying the following estimate
\begin{align}
\|\z Y\y^M\nabla^k\psi^1\|\lesssim1,\hspace{1mm}\text{  for  } 1\leqslant k \leqslant K, K \text{ and } M \text{ are large constants }.
\end{align}
\end{lemma}
\begin{proof}

We homogenize the boundary conditions of system (\ref{stream-euler1-bry}). Let
$$\tilde{\psi}=\psi^1-\frac{L-x}{L}\psi^1_0(Y)-\frac{x}{L}[\psi^1_L(Y)-\chi(Y)\int_0^Lv^{0}_b(X',0)\dd X']-\chi(Y)\int_0^Xv^{0}_b(X',0)\dd X',$$
here $\chi(Y)$ is a nonnegative smooth cut-off function, $\chi|_{[0,1]}=1$ and $\chi|_{[2,\infty]}=0$.  $\td{\psi}$ satisfies
\begin{align} \label{b euler}
\left\{
\begin{aligned}
     &\Delta \tilde{\psi}-F_e'(\psi)\tilde{\psi}=\tilde{F},\\
     &\tilde{\psi}|_{\p \Omega}=0.
\end{aligned}
\right.
\end{align}
Notice that 
$$\Delta\psi=F_e(\psi),$$ 
so $0<c_0\leqslant u_e^0=\psi_Y\leqslant C_0<\infty$ satisfies 
$$\Delta u^0_e=F'_e(\psi) u^0_e.$$
Let $w=\frac{\td{\psi}}{u^0_e}$, then 
$$
u^0_e\Delta w+2\nabla u^0_e\cdot\nabla w=\td{F}.
$$
Above equation times $u^0_e$, we show the equation (\ref{b euler}) is equivalent to 
\begin{align} \label{b imp}
\left\{
\begin{aligned}
     &\p_X[(u^0_e)^2w_X]+\p_Y[(u^0_e)^2w_Y]=u^0_e \tilde{F},\\
     &w|_{\p \Omega}=0.
\end{aligned}
\right. 
\end{align}
We can easily get a prior estimates of equation (\ref{b imp}). Multiply equation (\ref{b imp}) by $w$ and integrate in $\Omega$,
$$
\|u^0_e w_X\|^2+\|u^0_e w_Y\|^2=-\z w,u^0_e\td{F}\y\lesssim\|w\|\|F\|\lesssim\|w_X\|\|F\|.
$$
So we have 
\begin{align}\label{wuqu b}
\|\nabla w\|\lesssim\|\td{F}\|.
\end{align}
The equality (\ref{wuqu b}) actually shows the existence of solution about equation (\ref{b imp}). Moreover, if $\td{F}$ is a smooth function decaying fast when $Y\rightarrow\infty$, we can obtain the following estimate by the mathematical induction method: 
\begin{align}
\|\z Y\y^M\nabla^k w\|\lesssim1,\hspace{1mm}\text{  for  } 1\leqslant k \leqslant K, K \text{ and } M \text{ are large constants }.
\end{align}
So Lemma \ref{LEuler} is right.
\end{proof}

\begin{remark}
The boundary conditions of $\psi^1$ in (\ref{stream-euler1-bry}) imply the following boundary conditions of $[u^1_e,v^1_e]$
\begin{align}\label{bry of euler1}
\begin{aligned}
&u^1_e|_{X=0}=\p_Y\psi^1_0(Y),\hspace{3 mm}u^1_e|_{X=L}=\p_Y\psi^1_L(Y),\\
&v^1_e|_{Y=0}=-v^0_b(X,0),\hspace{3 mm}[u^1_e,v^1_e]|_{Y\rightarrow\infty}=0.
\end{aligned}
\end{align}
So we actually constructed a solution $[u^1_e,v^1_e]$ to equations (\ref{euler1}) with boundary conditions (\ref{bry of euler1}).
\end{remark}

Next we need to solve the first order boundary layer profile. For simplicity, we introduce some notations.
\begin{equation}
\begin{aligned}
&u^k_p:=u^k_b+\sum_{j=0}^k\frac{y^j}{j!}\p^j_Yu^{k-j}_e|_{Y=0},\hspace{5 mm} u^{(k)}_e:=\sum_{j=0}^k\frac{y^j}{j!}\p^j_Yu^{k-j}_e|_{Y=0},\\
&v^k_p:=v^k_b-v^k_b|_{y=0}+\sum_{j=0}^{k}\frac{y^{j+1}}{(j+1)!}\p^{j+1}_Yv^{k-j}_e|_{Y=0},\hspace{5 mm} v^{(k)}_e:=\sum_{j=0}^{k}\frac{y^{j+1}}{(j+1)!}\p^{j+1}_Yv^{k-j}_e|_{Y=0}.
\end{aligned}
\end{equation}
And $[u^1_b,v^1_b,p^1_b]$ solves the linearized Prandtl's equations around $[u^0_p,v^0_p]$:
\begin{align} \label{prandtl1}
\left
\{
\begin{aligned}
& u^0_p \p_x u^1_b + u^1_b \p_x u^0_p + \p_y u^0_p [v^1_b - v^1_b|_{y = 0}] + v^0_p \p_y u^1_b- \p_{yy} u^1_b + \p_x p^1_b = f^{(1)}, \\
& \p_y p^1_b = 0,\\
& \p_x u^1_b + \p_y v^1_b = 0,\\
& u^1_b|_{y = 0} = -u^1_e|_{Y = 0}, \hspace{5 mm} [u^1_b, v^1_b]|_{y \rightarrow \infty} = 0,
\end{aligned}
\right.
\end{align}
where
\begin{equation}
\begin{aligned}
f^{(1)}=&-\{u^0_bu^{(1)}_{ex}+u^0_{bx}u^{(1)}_e+v^0_b\p_yu^{(1)}_e+u^0_{by}v^{(1)}_e\}.
\end{aligned}
\end{equation}
We see that $f^{(1)}$ decays fast when $y\rightarrow\infty$ from Lemma \ref{pra0}. Since that above equations are linear parabolic type equations, we add the boundary condition on $u^1_b|_{x=0}$.
\begin{align} \label{prandtl1 bry}
\left
\{
\begin{aligned}
& u^0_p \p_x u^1_b + u^1_b \p_x u^0_p  + v^0_p \p_y u^1_b+ [v^1_b - v^1_b|_{y = 0}]\p_y u^0_p - \p_{yy} u^1_b+ \p_x p^1_b = f^{(1)}, \\
& \p_y p^1_b = 0,\\
& \p_x u^1_b + \p_y v^1_b = 0,\\
& u^1_b|_{x = 0}=U^1_B, \hspace{5 mm} u^1_b|_{y = 0} = -u^1_e|_{Y = 0}, \hspace{5 mm} [u^1_b, v^1_b]|_{y \rightarrow \infty} = 0.
\end{aligned}
\right.
\end{align}
Iyer in \cite{GI-H} proved the well-posedness of above system when $v^0_e=0$. In our case, $v^0_p$ is different because $v^0_p\sim yv^0_{eY}(x,0)$ as $y$ goes to $\infty$, still we have
\begin{lemma}\label{pra1}
If $U^1_B$, $f^{(1)}$ and their derivatives are bounded and decaying rapidly, they satisfy the parabolic compatibility conditions, then equations (\ref{prandtl1 bry}) admit a unique solution $[u^1_b,v^1_b]$, and
\begin{align}
\| \z y\y^M\nabla^k u^1_b\|_\infty+\| \z y\y^M\nabla^k v^1_b\|_\infty \lesssim 1  \hspace{2mm}\text{  for  }\hspace{2mm} 0\leqslant k \leqslant K,
\end{align}
where $K$ and $M$ are large constants.
\end{lemma}
\begin{proof}
For convenience, we write $[\ub,\vb]:=[u^0_p,v^0_p]$ and we homogenize the system (\ref{prandtl1 bry}) as the following way:
\begin{align}
\begin{aligned}
&u(x,y)=u^1_b(x,y)+u^1_e(x,0)\eta(y),\\
&v(x,y)=v^1_b(x,y)-v^1_b(x,0)+u^1_{eX}(x,0)I_\eta(y), \\
&I_\eta(y):=\int_y^\infty\eta(y')dy'.
\end{aligned}
\end{align}
Here, we select $\eta$ to be a $C^\infty$ function satisfying the following:
\begin{align}
\eta(0)=1,\hspace{5mm} \int_0^\infty\eta=0,\hspace{5mm} \eta\text{ decays fast as }y\rightarrow\infty.
\end{align}
Due to (\ref{prandtl1 bry}), the homogenized unknowns $[u,v]$ satisfy the system
\begin{align}
\left
\{
\begin{aligned}
& \ub \p_x u + u \p_x \ub + \vb \p_y u + v\p_y \ub - \p_{yy} u +p_x= f^{(1)}+F=: h, \\
& p_y=0,\\
& \p_x u + \p_y v = 0,\\
& u|_{x = 0}=U^1_B+u^1_e(0,0)\eta(y)=:u_0(y), \hspace{3 mm} [u,v]|_{y = 0} =0, \hspace{3 mm} u|_{y \rightarrow \infty} = 0,
\end{aligned}
\right.
\end{align}
where
\begin{align}
F=\ub u^1_{eX}(x,0)\eta+\ub_xu^1_e(x,0)\eta+\vb u^1_e(x,0)\eta'+\ub_y u^1_{eX}(x,0)I_\eta-u^1_e(x,0)\eta''.
\end{align}
Notice that $p$ is independent on $y$, we evaluate the equation as $y\rightarrow\infty$, we have $p_x=0$.
We still using the stream-function of $[u,v]$,
\begin{align}
\phi(x,y):=\int_0^y u(x,y')dy',\hspace{5 mm} \p_y\phi=u, \hspace{5 mm}\p_x\phi=-v.
\end{align}
Then $\phi$ satisfies
\begin{align}\label{phi}
\left
\{
\begin{aligned}
& \ub \phi_{xy} +\ub_x\phi_{y}+ \vb \phi_{yy}- \phi_x \ub_{y} - \phi_{yyy}  =  h, \\
& \phi|_{x = 0}=\int_0^yu_0(y')dy', \hspace{3 mm} \phi|_{y = 0} =\phi_y|_{y = 0}=0, \hspace{3 mm} \phi_y|_{y \rightarrow \infty} = 0.
\end{aligned}
\right.
\end{align}
In order to give a priori estimate of (\ref{phi}), we denote $g=\frac{\phi}{\ub}$. Recall $\ub\sim y$ when $y\leqslant1$ and $\ub\sim 1$, when $y\geqslant1$, and $\phi|_{y = 0} =\phi_y|_{y = 0}=0$, $g$ is well-defined. And $g$ satisfies
\begin{align}\label{g}
\left
\{
\begin{aligned}
&\p_{x}[\ub^2 g_y] - \p^3_{y}[\ub g]+\vb\p^2_{y}[\ub g]-\ub\vb_{yy}g  =  h, \\
& g|_{x = 0}=\frac{\int_0^yu_0}{\ub}, \hspace{3 mm} g|_{y = 0}=0, \hspace{3 mm} g_y|_{y \rightarrow \infty} = 0.
\end{aligned}
\right.
\end{align}
Now we define the norms of $g$:
\begin{align}\label{norm g}
\begin{aligned}
\|g\|_{\Xi_0}:=\sup_{0\leqslant x_0\leqslant L}\|\ub g_y\rho\fsx+\|\sqrt{\ub} g_{yy}\rho\fs,\\
\|g\|_{\Xi_1}:=\sup_{0\leqslant x_0\leqslant L}\|\ub g_{xy}\frac{\rho}{\z y\y}\fsx+\|\sqrt{\ub} g_{xyy}\frac{\rho}{\z y\y}\fs,
\end{aligned}
\end{align}
here $\rho=\z y\y^N$, for $N$ large constant. Next, let us prove the following priori estimate of $g$.
Suppose $g$ be a smooth solution of (\ref{g}), $L>0$ small enough, then
\begin{align}\label{g est1}
\|g\|_{\Xi_0}^2\lesssim &\|\ub g_y \rho\|_{L^2_y(x=0)}^2+\|h\rho\fs^2,\\\label{g est2}
\|g_x\|_{\Xi_1}^2\lesssim &\|\ub g_{xy} \frac{\rho}{\z y\y}\|_{L^2_y(x=0)}^2+\|g\|_{\Xi_0}^2+\|h_x\frac{\rho}{\z y\y}\fs^2.
\end{align}
Multiply equation (\ref{g}) by $g_y\rho^2$ and integrate in $(0,x_0)\times(0,\infty)$.
\[
\begin{split}
\jfxy[\ub^2 g_y]_x g_y\rho^2\dd x\dd y=&\jfxy \ub^2 g_{xy} g_y\rho^2\dd x\dd y+\jfxy2\ub\ub_x g^2_y\rho^2\dd x\dd y\\
                           =&\frac{1}{2}\|\ub g_y\rho\fsx^2-\frac{1}{2}\|\ub g_y\rho\|^2_{L^2_y(x=0)}+\jfxy\ub\ub_x g^2_y\rho^2\dd x\dd y.
\end{split}
\]
We can dominate $\|g_y\|$ by $\|g\|_{\Xi_0}$. Let $0<\xi\leqslant1$ be a constant being choosing later. $\chi(y)$ is smooth cut-off function, satisfies
$\chi|_{[0,1]}=1$, $\chi|_{[2,\infty]}=0$. Then,
\[
\begin{split}
\|g_y\rho\fs\lesssim\|g_y[1-\chi(\frac{y}{\xi})]\rho\fs+\|g_y\chi(\frac{y}{\xi})\rho\fs.
\end{split}
\]
When $y\leqslant1$, $1-\chi(\frac{y}{\xi})\lesssim\frac{y}{\xi}\lesssim\frac{\ub}{\xi}$, when $y>1$,  $1-\chi(\frac{y}{\xi})\lesssim\ub\lesssim\frac{\ub}{\xi}$. So
\[
\begin{split}
\|g_y[1-\chi(\frac{y}{\xi})]\rho\fs\lesssim\frac{1}{\xi^2}\|\ub g_y\rho\fs^2\lesssim \frac{L}{\xi^2}\|g\|_{\Xi_0}^2.
\end{split}
\]
And
\[
\begin{split}
\|g_y\chi(\frac{y}{\xi})\rho\fs^2=&-\jfxy2 yg_yg_{yy}\chi^2(\frac{y}{\xi})\rho^2 \dd x\dd y-\jfxy\frac{2}{\xi}yg^2_y\chi(\frac{y}{\xi})\chi'(\frac{y}{\xi})\rho^2\dd x\dd y\\
                                   &-\jfxy2 yg^2_y\chi^2(\frac{y}{\xi})\rho\rho_y\dd x\dd y\\
                                       \lesssim &\|y\chi(\frac{y}{\xi})g_{yy}\rho\fs^2+\frac{1}{\xi^2}\|\ub g_y\rho\fs^2\\
                                       \lesssim &\xi\|\sqrt{\ub}g_{yy}\rho\fs^2+\frac{L}{\xi^2}\|g\|_{\Xi_0}^2.
\end{split}
\]
So we have
\[
\begin{split}
\|g_y\rho\fs^2\lesssim \xi\|\sqrt{\ub}g_{yy}\rho\fs^2+\frac{L}{\xi^2}\|g\|_{\Xi_0}^2,
\end{split}
\]
select $\xi=L^\frac{1}{3}$, then
\begin{align}\label{g hardy}
\|g_y\rho\fs^2\lesssim L^\frac{1}{3}\|g\|_{\Xi_0}^2.
\end{align}
So the first term is
\begin{align}\label{g 1}
\jfxy[\ub^2 g_y]_x g_y\rho^2\dd x\dd y=\frac{1}{2}\|\ub g_y\rho\fsx^2-\frac{1}{2}\|\ub g_y\rho\|_{L^2_y(x=0)}^2+O\big(L^\frac{1}{3}\|g\|_{\Xi_0}^2\big).
\end{align}
The second term:
\[
\begin{split}
-\jfxy\p^3_y[\ub g]g_y\rho^2\dd x\dd y=&\jfxy\p^2_y[\ub g]g_{yy}\rho^2\dd x\dd y+\jfxy 2\p^2_y[\ub g]g_{y}\rho_y\rho \dd x\dd y\\
&+\z\p^2_y[\ub g],g_{y}\rho^2\yy.
\end{split}
\]
\[
\begin{split}
\jfxy\p^2_y[\ub g]g_{yy}\rho^2\dd x\dd y=&\jfxy(\ub g_{yy}+2\ub_y g_y+\ub_{yy}g)g_{yy}\rho^2\dd x\dd y\\
                              =&\|\sqrt{\ub}g_{yy}\rho\fs^2-\z\ub_yg_y,g_y\yy+\jfxy(\ub_y\rho^2)_y g^2_y\dd x\dd y\\
                               &-\jfxy\ub_{yy}g_y^2\rho^2\dd x\dd y-\jfxy(\ub_{yy}\rho^2)_yg g_y\dd x\dd y\\
                              =&\|\sqrt{\ub}g_{yy}\rho\fs^2-\z\ub_yg_y,g_y\yy\\
                              &+O\big(\|g_y\rho\fs^2+\|y(\ub_{yy}\rho^2)_y\|_{L^\infty}\|\frac{g}{y}\fs \|g_y\fs\big)\\
                              =&\|\sqrt{\ub}g_{yy}\rho\fs^2-\z\ub_yg_y,g_y\yy+O\big(\|g_y\rho\fs^2\big),
\end{split}
\]
\[
\begin{split}
\jfxy2\p^2_y[\ub g]g_{y}\rho_y\rho \dd x\dd y=&\jfxy\ub (g_y^2)_y \rho_y\rho \dd x\dd y+\jfxy4\ub_yg_y^2\rho_y\rho \dd x\dd y\\
                                    &+\jfxy2\ub_{yy}gg_y\rho_y\rho \dd x\dd y\\
                                   =&O\big(\|g_y\rho\fs^2+\|y\ub_{yy}\rho_y\rho\|_{L^\infty}\|\frac{g}{y}\fs \|g_y\fs\big))\\
                                   =&O\big(\|g_y\rho\fs^2\big),
\end{split}
\]
\[
\begin{split}
\z\p^2_y[\ub g],g_{y}\rho^2\yy=2\z\ub_yg_y,g_y\rho^2\yy.
\end{split}
\]
So the second term is
\begin{align}\label{g 2}
-\jfxy\p^3_y[\ub g]g_y\rho^2\dd x\dd y=\|\sqrt{\ub}g_{yy}\rho\fs^2+\z\ub_yg_y,g_y\rho^2\yy+O\big(L^\frac{1}{3}\|g\|_{\Xi_0}^2\big).
\end{align}
The third term is
\begin{align}\label{g 3}
\begin{aligned}
\jfxy\vb(\ub g)_{yy}g_y\rho^2\dd x\dd y=&\jfxy\vb(\ub g_{yy}+2\ub_y g_y+\ub_{yy}g)g_y\rho^2\dd x\dd y\\
                             =&-\frac{1}{2}\jfxy(\vb\ub\rho^2)_y g^2_y\dd x\dd y+O\big(\|\vb \ub_y \|_{L^\infty}\|g_y\rho\fs^2\\
                              &+\|y\vb\ub_{yy}\rho^2\|_{L^\infty}\|\frac{g}{y}\fs\|g_y\fs\big)\\
                             =&O\big(\|\frac{(\vb\ub\rho^2)_y}{\rho^2} \|_{L^\infty}\|g_y\rho\fs^2+\|g_y\rho\fs^2\big)\\
                             =&O\big(L^\frac{1}{3}\|g\|_{\Xi_0}^2\big).
\end{aligned}
\end{align}
And the last one is
\begin{align}\label{g 4}
\begin{aligned}
\jfxy\vb_{yy}\ub g g_y\rho^2 \dd x\dd y=O\big(\|y\vb_{yy}\ub\rho^2\|_{L^\infty}\|\frac{g}{y}\fs\| g_y\fs\big)=O\big(\|g_y\rho\fs^2\big).
\end{aligned}
\end{align}
Collect (\ref{g 1}), (\ref{g 2}), (\ref{g 3}), (\ref{g 4}), we have
\begin{align}\label{g what?}
\begin{aligned}
&\frac{1}{2}\|\ub g_y\rho\fsx^2+\|\sqrt{\ub}g_{yy}\rho\fs^2+\z\ub_yg_y,g_y\rho^2\yy=\\
&O\big(L^\frac{1}{3}\|g\|_{\Xi_0}^2\big)+\frac{1}{2}\|\ub g_y\rho\|_{L^2_y(x=0)}^2+\jfxy hg_y\rho^2\dd x\dd y.
\end{aligned}
\end{align}
Take the supremum of $0\leqslant x_0\leqslant L$, notice that $L$ small enough,
\begin{align}\label{g es}
\begin{aligned}
\sup_{0\leqslant x_0\leqslant L}\|\ub g_y\rho\fsx^2+\|\sqrt{\ub}g_{yy}\rho\fs^2+\z\ub_yg_y,g_y\rho^2\yy\lesssim\|\ub g_y\rho\|^2_{L^2_y(x=0)}+\| h\rho\fs^2.
\end{aligned}
\end{align}

The inequality (\ref{g est2}) is similar to the (\ref{g est1}). Differential equation (\ref{g}) respect to $x$,
\begin{align}\label{g x}
\begin{aligned}
&\p_{x}[\ub^2 g_{xy}]- \p^3_{y}[\ub g_x]+\vb\p^2_{y}(\ub g_x)-\ub\vb_{yy}g_x+\p_{x}[2\ub\ub_x g_{y}]- \p^3_{y}[\ub_x g]\\
&+\vb_x\p^2_{y}(\ub g)+\vb\p^2_{y}(\ub_x g)-\ub_x\vb_{yy}g-\ub\vb_{xyy}g =  h_x.
\end{aligned}
\end{align}
Take $g_{xy}\frac{\rho^2}{\z y\y^2}$ as the test function, like (\ref{g what?}),
\begin{align}\label{g x1}
\begin{aligned}
 &\jfxy[\p_{x}[\ub^2 g_{xy}]- \p^3_{y}[\ub g_x]+\vb\p^2_{y}(\ub g_x)-\ub\vb_{yy}g_x] g_{xy}\frac{\rho^2}{\z y\y^2} \dd x\dd y\\
=&\frac{1}{2}(\|\ub g_{xy}\frac{\rho}{\z y\y}\fsx^2-\|\ub g_{xy}\frac{\rho}{\z y\y}\|^2_{L^2_y(x=0)})+\|\sqrt{\ub}g_{xyy}\frac{\rho}{\z y\y}\fs^2\\
&+\z\ub_yg_{xy},g_{xy}\frac{\rho}{\z y\y}\yy+O\big(L^\frac{1}{3}\|g\|_{\Xi_1}^2\big).
\end{aligned}
\end{align}
And
\begin{align}\label{g x2}
\begin{aligned}
  &\jfxy[\p_{x}(2\ub\ub_x g_{y})+\vb_x\p^2_{y}(\ub g)+\vb\p^2_{y}(\ub_x g)-\ub_x\vb_{yy}g-\ub\vb_{xyy}g]g_{xy}\frac{\rho^2}{\z y\y^2}\dd x\dd y\\
 =&O\big(\|\sqrt{\ub} g_{yy}\rho\fs^2+\|g_y\rho\fs^2+\|g_{xy}\frac{\rho}{\z y\y}\fs^2\big)\\
 =&O\big(\|g\|_{\Xi_0}^2+L^\frac{1}{3}\|g\|_{\Xi_1}^2\big).
\end{aligned}
\end{align}
The difficult term is
\begin{align}\label{g dif}
\begin{aligned}
\jfxy- \p^3_{y}[\ub_x g]g_{xy}\frac{\rho^2}{\z y\y^2}\dd x\dd y=&-\jfxy(\ub_xg_{yyy}+3\ub_{xy}g_{yy})g_{xy}\frac{\rho^2}{\z y\y^2}\dd x\dd y\\
&+O\big(\|g_y\rho\fs^2+\|g_{xy}\frac{\rho}{\z y\y}\fs^2\big)\\
                                                     =&O\big((\|\ub g_{yyy}\frac{\rho}{\z y\y}\fs+\|g_{yy}\frac{\rho}{\z y\y}\fs)\|g_{xy}\frac{\rho}{\z y\y}\fs\\
                                                     &+\|g_y\rho\fs^2+\|g_{xy}\frac{\rho}{\z y\y}\fs^2 \big).
\end{aligned}
\end{align}
From equation (\ref{phi}), we have
\[
\begin{split}
\|\phi_{yyy}\frac{\rho}{\z y\y}\fs^2=O\big(\|g\|_{\Xi_0}^2+L^\frac{1}{3}\|g\|_{\Xi_1}^2+\|h\frac{\rho}{\z y\y}\fs^2\big),
\end{split}
\]
notice that the fact
\[
\begin{split}
\|\chi\p^2_y(\frac{\phi}{y})\fs=O\big(\|\phi_{yyy}\fs+\|\phi_{yy}\fs\big),
\end{split}
\]
similarly, we can get
\[
\begin{split}
\|\chi g_{yy}\fs=O\big(\|\phi_{yyy}\fs+\|\phi_{yy}\fs\big),
\end{split}
\]
so we have
\[
\begin{split}
\|g_{yy}\frac{\rho}{\z y\y}\fs^2=&\|g_{yy}\chi\fs^2+\|g_{yy}(1-\chi)\frac{\rho}{\z y\y}\fs^2\\
            =&O\big(\|\phi_{yyy}\fs^2+\|\phi_{yy}\fs^2+\|\sqrt{\ub}g_{yy}\frac{\rho}{\z y\y}\fs^2\big),\\
            =&O\big(\|g\|_{\Xi_0}^2+L^\frac{1}{3}\|g\|_{\Xi_1}^2+\|h\frac{\rho}{\z y\y}\fs^2\big),
\end{split}
\]
and
\[
\begin{split}
\|\ub g_{yyy}\frac{\rho}{\z y\y}\fs^2=&\|(\phi_{yyy}-3\ub_y g_{yy}-3\ub_{yy}g_{y}-\ub_{yyy}g)\frac{\rho}{\z y\y}\fs^2\\
                                    =&O\big(\|g\|_{\Xi_0}^2+L^\frac{1}{3}\|g\|_{\Xi_1}^2+\|h\frac{\rho}{\z y\y}\fs^2\big).
\end{split}
\]
We conclude (\ref{g dif}) as
\begin{align}\label{g x3}
\begin{aligned}
\jfxy- \p^3_{y}[\ub_x g]g_{xy}\frac{\rho^2}{\z y\y^2}\dd x\dd y=O\big(\|g\|_{\Xi_0}^2+L^\frac{1}{3}\|g\|_{\Xi_1}^2+\|h\frac{\rho}{\z y\y}\fs^2\big).
\end{aligned}
\end{align}
Collect (\ref{g x1}), (\ref{g x2}), (\ref{g x3}), we have
\begin{align}\label{gx what?}
\begin{aligned}
&\frac{1}{2}\|\ub g_{xy}\frac{\rho}{\z y\y}\fsx^2+\|\sqrt{\ub}g_{xyy}\frac{\rho}{\z y\y}\fs^2+\z\ub_yg_{xy},g_{xy}\frac{\rho}{\z y\y}\yy\\
=&\frac{1}{2}\|\ub g_{xy}\frac{\rho}{\z y\y}\|^2_{L^2_y(x=0)}+O\big(\|g\|_{\Xi_0}^2+L^\frac{1}{3}\|g\|_{\Xi_1}^2+\|h_x\frac{\rho}{\z y\y}\fs^2+\|h\frac{\rho}{\z y\y}\fs^2\big).
\end{aligned}
\end{align}
So we get the inequalities  (\ref{g est1}) and (\ref{g est2}). These inequalities show if $g$ satisfies the linear parabolic type equation (\ref{g}), then $\|g\|_{\Xi_0}$ and $\|g\|_{\Xi_1}$ can be dominated by its initial data and $h$, we can use the standard method to prove the local existence of solution. Follow this way, we can also get the high order derivatives estimates to show the smoothness of solution.
\end{proof}
In fact, the system (\ref{g}) admits a smooth solution $g$ even if $L$ is large, since the local well-posedness means the global well-posedness for linear parabolic type equation.

The second order Euler profile $[u^2_e,v^2_e,p^2_e]$ solves the linearized Euler equations around $[u^0_e,v^0_e]$ with the force terms:
\begin{align} \label{euler2}
\left\{
\begin{aligned}
&u^0_eu^2_{eX}+u^0_{eX}u^2_e+v^0_eu^2_{eY}+u^0_{eY}v^2_e+p^2_{eX}=F^{(2)},\\
&u^0_ev^2_{eX}+v^0_{eX}u^2_e+v^0_ev^2_{eY}+v^0_{eY}v^2_e+p^2_{eY}=G^{(2)},\\
&\p_X u^2_e + \p_Y v^2_e = 0,\\
&v^2_e|_{Y=0}=-v^1_b|_{y=0},\\
\end{aligned}
\right.
\end{align}
where
\begin{equation}
\begin{aligned}
&F^{(2)}=-(u^{1}_eu^1_{ex}+v^{1}_eu^1_{eY})+\Delta u^{0}_e,\\
&G^{(2)}=-(u^{1}_ev^1_{ex}+v^{1}_ev^1_{eY})+\Delta v^{0}_e.
\end{aligned}
\end{equation}
We can treat above equations as that of the first order Euler flow.
\begin{align} \label{stream-euler2}
\begin{aligned}
     \p_\theta[\Delta_{XY} \psi^2-F_e'(\psi)\psi^2-\frac{F''_e(\psi)}{2}(\psi^1)^2]=\frac{\Delta^2_{XY}\psi}{u^0_e}.
\end{aligned}
\end{align}
Let $$H(\theta,\psi)=\int_0^\theta \frac{\Delta^2_{XY}\psi(\theta',Y(\theta',\psi))}{u_e^0(\theta',\psi)}\dd\theta'.$$ Notice that $\psi\sim Y$ when $Y\rightarrow\infty$, we have that $H$ is of fast decay as $\psi\rightarrow\infty$ because of (\ref{Eul profile}).
We can find a solution of the following equations
\begin{align} \label{stream-euler2-bry}
\left\{
\begin{aligned}
     &\Delta_{XY} \psi^2-F_e'(\psi)\psi^2-\frac{F''_e(\psi)}{2}(\psi^1)^2=H(\theta(X,Y),\psi(X,Y)),\\
     &\psi^2|_{X=0}=\psi^2_0(Y),\hspace{3 mm}\psi^2|_{X=L}=\psi^2_L(Y),\\
     &\psi^2|_{Y=0}=\int_0^Xv^1_b(X',0)\dd X',\hspace{3 mm}\psi^2|_{Y\rightarrow\infty}=0,
\end{aligned}
\right.
\end{align}
with suitable $\psi^2_0(Y)$, $\psi^2_L(Y)$, and we have the estimate of the second order Euler flow
\begin{align}
\|\z Y\y^M\nabla^k\psi^2\|\lesssim1,\hspace{1mm}\text{  for  } 1\leqslant k \leqslant K,\quad K \text{ and } M \text{ large comstants }.
\end{align}

The second order boundary layer profile $[u^2_b,v^2_b,p^2_b]$ is similar to the first, we need to solve the following equations
\begin{align} \label{prandtl2 bry}
\left
\{
\begin{aligned}
& u^0_p \p_x u^2_b + u^2_b \p_x u^0_p + v^0_p \p_y u^2_b+  [v^2_b - v^2_b|_{y = 0}]\p_y u^0_p- \p_{yy} u^2_b + \p_x p^2_b = f^{(2)}, \\
& \p_y p^2_b = g^{(2)},\\
& \p_x u^2_b + \p_y v^2_b = 0,\\
& u^2_b|_{x = 0}=U^2_B, \hspace{5 mm} u^2_b|_{y = 0} = -u^2_e|_{Y = 0}, \hspace{5 mm} [u^2_b, v^2_b]|_{y \rightarrow \infty} = 0.
\end{aligned}
\right.
\end{align}
Where
\begin{equation}
\begin{aligned}
f^{(2)}=&-\{u^0_bu^{(2)}_{ex}+u^0_{bx}u^{(2)}_e+v^0_bu^{(2)}_{ey}+u^0_{by}v^{(2)}_e\\
        &+u^{1}_pu^1_{bx}+u^{1}_bu^{(1)}_{ex}+v^{1}_pu^1_{by}+v^{1}_bu^{(1)}_{ey}-u^{0}_{bxx}\},\\
g^{(2)}=&-\{u^{0}_bv^0_{px}+u^{(0)}_ev^0_{bx}+v^{0}_bv^0_{py}+(v^{(0)}_e+v^1_e|_{Y=0})v^0_{by}\\
        &-v^{0}_{byy}\}.\\
\end{aligned}
\end{equation}
We can see $f^{(2)}$ and $g^{(2)}$ decays fast when $y\rightarrow\infty$ from Lemma \ref{pra0} and Lemma \ref{pra1}. We can solve $p^2_b(x,y)=-\int_y^\infty g^{(2)}(x,y')\dd y'$. By using the same argument of Lemma \ref{pra1}, we have
\begin{align}
\| \z y\y^M\nabla^ku^2_b\|_\infty+\| \z y\y^M\nabla^k v^2_b \|_\infty \lesssim 1 \hspace{2mm}\text{  for  }\hspace{2mm} 0\leqslant k \leqslant K,
\end{align}
where $K$ and $M$ are large constants.

After that, $p_b^{3}$ is solved by
\begin{equation}\label{p3}
\begin{aligned}
p_b^{3}=&\int_y^\infty\{\sum_{j=0}^{1}[u^{1-j}_bv^j_{px}+u^{(1-j)}_ev^j_{bx}+v^{1-j}_bv^j_{py}\\
          &+(v^{(1-j)}_e+v^{2-j}_e|_{Y=0})v^j_{by}]-v^{1}_{byy}\}\dd y'.
\end{aligned}
\end{equation}
We can conclude the following proposition for the approximate profiles.
\begin{proposition}\label{construct} Under the assumptions of Theorem \ref{main} then equations (\ref{euler1}), (\ref{prandtl1}), (\ref{euler2}), (\ref{prandtl2 bry}) admit smooth solutions $[u^j_e,v^j_e]$, $[u^j_b,v^j_b]$ for $j=1,2$, and  the following estimates hold
\begin{align}
\begin{aligned}
&\| \z y\y^M\nabla^k u^j_b\|_\infty+\| \z y\y^M\nabla^k v^j_b\|_\infty \lesssim 1\hspace{2mm}\text{  for  }\hspace{2mm}  0\leqslant k \leqslant K,j=0,1,2, \\
&\| \z Y\y^M\nabla^ku^j_e\|_\infty+\| \z Y\y^M\nabla^kv^j_e\|_\infty  \lesssim 1 \hspace{2mm}\text{  for  }\hspace{2mm}  0\leqslant k \leqslant K,  j=1,2,
\end{aligned}
\end{align}
\noindent where $K,$ $M$ sufficiently large constants, $\z y\y=y+1$ and $\z Y\y=Y+1$.
\end{proposition}

Notices that $v^2_b|_{y=0}\neq0$.  We need to match the boundary conditions at $y=0$, and also   $v^2_b|_{y \rightarrow\infty}=0$. Then we can modify $[\hat{u}^2_b,\hat{v}^2_b]$ in this way:
\begin{equation}\label{modify}
\begin{aligned}
&\hat{u}^2_b(x,y):=\chi(\sqrt\e y)u^2_b(x,y)-\sqrt\e\chi'(\sqrt\e y)\int_0^y u^2_b(x,y')\dd y',\\
&\hat{v}^2_b(x,y):=\chi(\sqrt\e y)(v^2_b(x,y)-v^2_b(x,0)),
\end{aligned}
\end{equation}
where $\chi$ is a cut-off function satisfying $\chi|_{[0,1]}=1$ and $\chi|_{[2,\infty)}=0$.

And let $[U_s,V_s,P_s]$ be
\begin{equation}
\begin{aligned}
U_s(X,Y)=&u^0_e(X,Y)+u^0_b(X,\frac{Y}{\se})+\se[u^1_e(X,Y)+u^1_b(X,\frac{Y}{\se})]\\
          &+\e[u^2_e(X,Y)+\hat{u}^2_b(X,\frac{Y}{\se})],\\
V_s(X,Y)=&v^0_e(X,Y)+\se[v^0_b(X,\frac{Y}{\se})+v^1_e(X,Y)]+\e[v^1_b(X,\frac{Y}{\se})+v^2_e(X,Y)]\\
          &+\e^\frac{3}{2}\hat{v}^2_b(X,\frac{Y}{\se}),\\
P_s(X,Y)=&p^0_e(X,Y)+p^0_b(X,\frac{Y}{\se})+\se[p^1_e(X,Y)+p^1_b(X,\frac{Y}{\se})]\\
          &+\e[p^2_e(X,Y)+p^2_b(X,\frac{Y}{\se})]+\e^\frac{3}{2}p^3_b(X,\frac{Y}{\se}).
\end{aligned}
\end{equation}
Then the errors
\begin{equation}\label{R1,R2}
\begin{aligned}
&R_1:=U_sU_{sX}+V_sU_{sY}-\e\Delta U_s+P_{sX},\\
&R_2:=U_sV_{sX}+V_sV_{sY}-\e\Delta V_s+P_{sY},
\end{aligned}
\end{equation}
satisfy
\begin{equation}
\begin{aligned}\label{R order}
\|R_1\|+\|R_2\|\lesssim\e^{\frac{3}{2}}.
\end{aligned}
\end{equation}

\noindent \textbf{Acknowledgements:} L. Zhang is partially supported by NSFC under grant
11471320 and 11631008.

\bibliographystyle{springer}
\bibliography{mrabbrev,literatur}
\newcommand{\noopsort}[1]{} \newcommand{\printfirst}[2]{#1}
\newcommand{\singleletter}[1]{#1} \newcommand{\switchargs}[2]{#2#1}

\end{document}